\documentclass[11pt,reqno]{article}
\usepackage{latexsym, amsmath, amssymb, amsthm, a4, epsfig}
\usepackage{graphicx}

\newtheorem{theorem}{Theorem}[section]

\newtheorem{lemma}[theorem]{Lemma}
\newtheorem{prop}[theorem]{Proposition}

\newcommand{\nm}{\noalign{\smallskip}}
\newcommand{\ds}{\displaystyle}
\newcommand{\p}{\partial}

\newcommand{\eqnref}[1]{(\ref {#1})}

\newcommand{\Cbb}{\mathbb{C}}

\newcommand{\Nbb}{\mathbb{N}}

\newcommand{\Rbb}{\mathbb{R}}

\newcommand{\la}{\langle}
\newcommand{\ra}{\rangle}

\newcommand{\Ccal}{\mathcal{C}}

\newcommand{\Hcal}{\mathcal{H}}
\newcommand{\Lcal}{\mathcal{L}}
\newcommand{\Kcal}{\mathcal{K}}

\newcommand{\Rcal}{\mathcal{R}}
\newcommand{\Scal}{\mathcal{S}}
\newcommand{\Tcal}{\mathcal{T}}

\def\Ba{{\bf a}}
\def\Bb{{\bf b}}
\def\Bc{{\bf c}}

\def\Bf{{\bf f}}
\def\Bg{{\bf g}}

\def\Bn{{\bf n}}

\def\Bu{{\bf u}}
\def\Bv{{\bf v}}

\def\Bx{{\bf x}}
\def\By{{\bf y}}
\def\Bz{{\bf z}}
\def\BA{{\bf A}}
\def\BB{{\bf B}}
\def\BC{{\bf C}}
\def\BD{{\bf D}}

\def\BF{{\bf F}}

\def\BI{{\bf I}}

\def\BK{{\bf K}}

\def\BS{{\bf S}}
\def\BT{{\bf T}}
\def\BU{{\bf U}}

\newcommand{\Ga}{\alpha}
\newcommand{\Gb}{\beta}
\newcommand{\Gd}{\delta}
\newcommand{\Ge}{\epsilon}

\newcommand{\Gf}{\phi}
\newcommand{\Gvf}{\varphi}
\newcommand{\Gg}{\gamma}

\newcommand{\Gk}{\kappa}

\newcommand{\Gl}{\lambda}
\newcommand{\Gn}{\eta}
\newcommand{\Gm}{\mu}
\newcommand{\Gv}{\nu}
\newcommand{\Gp}{\pi}
\newcommand{\Gt}{\theta}

\newcommand{\Gr}{\rho}

\newcommand{\Gs}{\sigma}

\newcommand{\Go}{\omega}
\newcommand{\Gx}{\xi}
\newcommand{\Gy}{\psi}
\newcommand{\Gz}{\zeta}
\newcommand{\GD}{\Delta}

\newcommand{\GG}{\Gamma}

\newcommand{\GO}{\Omega}

\newcommand{\BGG}{{\bf \GG}}
\newcommand{\BGf}{\mbox{\boldmath $\Gf$}}
\newcommand{\BGvf}{\mbox{\boldmath $\Gvf$}}
\newcommand{\Bpsi}{\mbox{\boldmath $\Gy$}}
\newcommand{\wBS}{\widetilde{\BS}}
\newcommand{\wGl}{\widetilde{\Gl}}
\newcommand{\wGm}{\widetilde{\Gm}}
\newcommand{\wGv}{\widetilde{\Gv}}
\newcommand{\wBU}{\widetilde{\BU}}

\newcommand{\beq}{\begin{equation}}
\newcommand{\eeq}{\end{equation}}

\def\ol{\overline}
\newcommand{\hatna}{\widehat{\nabla}}

\numberwithin{equation}{section}
\numberwithin{figure}{section}

\begin{document}

\title{Spectral properties of the Neumann-Poincar\'e operator and cloaking by anomalous localized resonance for the elasto-static system\thanks{\footnotesize This work is supported by the Korean Ministry of Education, Sciences and Technology through NRF grants Nos. 2010-0017532 (to H.K) and 2012003224 (to S.Y)}}

\author{Kazunori Ando\thanks{Department of Mathematics, Inha University, Incheon
402-751, S. Korea (ando@inha.ac.kr, 22151063@inha.edu, hbkang@inha.ac.kr, kskim@inha.ac.kr)} \and Yong-Gwan Ji\footnotemark[2] \and Hyeonbae Kang\footnotemark[2] \and Kyoungsun Kim\footnotemark[2] \and Sanghyeon Yu\thanks{Seminar for Applied Mathematics, ETH Z\"urich, R\"amistrasse 101, CH-8092 Z\"urich, Switzerland (sanghyeon.yu@sam.math.ethz.edu)} }

\maketitle

\begin{abstract}
We first investigate spectral properties of the Neumann-Poincar\'e (NP) operator for the Lam\'{e} system of elasto-statics. We show that the elasto-static NP operator can be symmetrized in the same way as that for Laplace operator. We then show that even if elasto-static NP operator is not compact even on smooth domains, its spectrum consists of eigenvalues which accumulates to two numbers determined by Lam\'e constants. We then derive explicitly eigenvalues and eigenfunctions on disks and ellipses. We then investigate resonance occurring at eigenvalues and anomalous localized resonance at accumulation points of eigenvalues. We show on ellipses that cloaking by anomalous localized resonance takes place at accumulation points of eigenvalues.
\end{abstract}



\tableofcontents

\section{Introduction}

The Neumann-Poincar\'e  (NP) operator is a boundary integral operator which appears naturally when solving classical boundary value problems for the Laplace equation using layer potentials. Recently there is rapidly growing interest in the spectral properties of the NP operator in relation to plasmonics and cloaking by anomalous localized resonance. Plasmon resonance and anomalous localized resonance occur at eigenvalues and at the accumulation point of eigenvalues, respectively (see for example \cite{ACKLM-ARMA-13, MFZ-PR-05} and references therein). We emphasize that the spectral nature of the NP operator differs depending on smoothness of the domain on which the NP operator is defined. If the domain has a smooth boundary, $C^{1, \Ga}$ for some $\Ga>0$ to be precise, then the NP operator is compact on $L^2$ or $H^{-1/2}$ space. Since the NP operator can be realized as a self-adjoint operator by introducing a new inner product (see \cite{Kang-SMF-15, KPS-ARMA-07}), its spectrum consists of eigenvalues converging to $0$. If the domain has a corner, the corresponding NP operator may exhibit a continuous spectrum. For this and recent development of spectral theory of the NP operator for the Laplace operator we refer to \cite{KLY} and references therein.

The purpose of this paper is twofold. We first extend the spectral theory of the NP operator for the Laplace operator to that for the Lam\'e system of elasto-statics, and then investigate resonance and cloaking by anomalous localized resonance.

To describe results of this paper in a precise manner, we first introduce some notation.
Let $\GO$ be a bounded domain in $\Rbb^d$ ($d=2,3$) with the Lipschitz boundary, and let $(\Gl, \Gm)$ be the Lam$\acute{\rm e}$ constants for $\GO$ satisfying the strong convexity condition
\beq\label{strongcon}
\Gm > 0 \quad \mbox{and} \quad d\Gl + 2\Gm > 0.
\eeq
The isotropic elasticity tensor $\Cbb = ( C_{ijkl} )_{i, j, k, l = 1}^d$ and the corresponding elastostatic system $\Lcal_{\Gl,\Gm}$ are defined by
\beq
C_{ijkl} := \Gl \, \Gd_{ij} \Gd_{kl} + \mu \, ( \Gd_{ik} \Gd_{jl} + \Gd_{il} \Gd_{jk} )
\eeq
and
\beq
\Lcal_{\Gl,\Gm} \Bu := \nabla \cdot \Cbb \hatna \Bu = \Gm \Delta \Bu + (\Gl + \Gm) \nabla \nabla \cdot \Bu,
\eeq
where $\hatna$ denotes the symmetric gradient, namely,
\beq
\hatna \Bu := \frac{1}{2} \left( \nabla \Bu + \nabla \Bu^T \right). \nonumber
\eeq
Here $T$ indicates the transpose of a matrix. The corresponding conormal derivative on $\p \GO$ is defined to be
\beq\label{conormal}
\p_\nu \Bu := (\Cbb \hatna \Bu) \Bn = \Gl (\nabla \cdot \Bu) \Bn + 2\Gm (\hatna \Bu) \Bn \quad \mbox{on } \p \GO,
\eeq
where $\Bn$ is the outward unit normal to $\p \GO$.

Let $\BGG = \left( \Gamma_{ij} \right)_{i, j = 1}^d$ is the Kelvin matrix of fundamental solutions to the Lam\'{e} operator $\Lcal_{\Gl, \mu}$, namely,
\beq\label{Kelvin}
  \Gamma_{ij}(\Bx) =
  \begin{cases}
    - \ds \frac{\Ga_1}{4 \pi} \frac{\Gd_{ij}}{|\Bx|} - \frac{\Ga_2}{4 \pi} \displaystyle \frac{x_i x_j}{|\Bx|^3}, & \text{ if } d = 3, \\
    \nm
    \ds \frac{\Ga_1}{2 \pi} \Gd_{ij} \ln{|\Bx|} - \frac{\Ga_2}{2 \pi} \displaystyle \frac{x_i x_j}{|\Bx| ^2}, & \text{ if } d = 2,
  \end{cases}
\eeq
where
\beq
  \Ga_1 = \frac{1}{2} \left( \frac{1}{\mu} + \frac{1}{2 \mu + \Gl} \right) \quad\mbox{and}\quad \Ga_2 = \frac{1}{2} \left( \frac{1}{\mu} - \frac{1}{2 \mu + \Gl} \right).
\eeq
The NP operator for the Lam\'e system is defined by
\beq\label{BK}
\BK [\Bf] (\Bx) := \mbox{p.v.} \int_{\p \GO} \p_{\nu_\By} {\bf \GG} (\Bx-\By) \Bf(\By) d \Gs(\By) \quad \mbox{a.e. } \Bx \in \p \GO.
\eeq
Here, p.v. stands for the Cauchy principal value, and the conormal derivative $\p_{\nu_\By}\BGG (\Bx-\By)$ of the Kelvin matrix with respect to $\By$-variables is defined by
\beq\label{kerdef}
\p_{\nu_\By}\BGG (\Bx-\By) \Bb = \p_{\nu_\By} (\BGG (\Bx-\By) \Bb)
\eeq
for any constant vector $\Bb$.

We show that the NP operator $\BK$ can be realized as a self-adjoint operator on $H^{1/2}(\p\GO)^d$ ($H^{1/2}$ is a Sobolev space) by introducing a new inner product in a way parallel to the case of the Laplace operator. But, there is a significant difference between NP operators for the Laplace operator and the Lam\'e operator. The NP operator for the Lam\'e operator is {\it not} compact even if the domain has a smooth boundary (this was observed in \cite{DKV-Duke-88} correcting an error in \cite{Kup-book-65}), which means that we cannot infer directly that the NP operator has point spectrum (eigenvalues). However, we are able to show in this paper that the elasto-static NP operator on planar domains with $C^{1, \Ga}$ boundaries has only point spectrum. In fact, we show that on such domains
\beq\label{statcompact}
\BK^2 - k_0^2 I \ \mbox{ is compact},
\eeq
where
\beq\label{Gkdef}
k_0 = \frac{\mu}{2(2\mu+\Gl)}.
\eeq
It is worth mentioning that we are able to prove \eqnref{statcompact} only in two dimensions, and it is not clear if it is true in three dimensions. Probably there is a polynomial $p$ such that $p(\BK)$ is compact.
As an immediate consequence of \eqnref{statcompact}, we show that the spectrum of $\BK$ consists of eigenvalues which accumulate at $k_0$ and $-k_0$.
We then explicitly compute eigenvalues of $\BK$ on disks and ellipses. It turns out that $k_0$ and $-k_0$ are eigenvalues of infinite multiplicities (there are two other eigenvalues of finite multiplicities) on disks, while on ellipses $k_0$ and $-k_0$ are accumulation points of eigenvalues, but not eigenvalues, and the rates of convergence to $k_0$ and $-k_0$ are exponential.

Using the spectral properties of the NP operator we investigate resonance, especially cloaking by anomalous localized resonance (CALR). CALR on dielectric plasmonic material was first discovered in \cite{MN-PRSA-06}. It is shown that if we coat a dielectric material of circular shape by a plasmonic material of negative dielectric constant (with a dissipation), then huge resonance occurs and the polarizable dipole is cloaked when it is within the cloaking region. This result has been extended to general sources other than polarizable dipole sources \cite{ACKLM-ARMA-13, KLSW-CMP-14}. It is also shown in \cite{AK14} that CALR occurs not only on the coated structure, but also on ellipses.

In this paper we show that CALR also occurs on elastic structures. We consider an ellipse $\GO$ embedded in $\Rbb^2$, where the Lam\'e constants of the background are $(\Gl,\Gm)$ and those of $\GO$ are $(c+i\Gd)(\Gl,\Gm)$. Here $c$ is a negative constant and $\Gd$ is a loss parameter which tends to $0$. So, $\GO$ represents an elastic material with negative Lam\'e constants. Discussion on existence of such materials is beyond the scope of this paper. However, we refer to \cite{LLBW-Nature-01} for existence (in composites) of negative stiffness material, and to \cite{KM-JMPS-14} for effective properties.  We show that if $c$ satisfies
\beq\label{kzero}
k(c):=\frac{c+1}{2(c-1)}= k_0 \ \text{ or } \ -k_0,
\eeq
then CALR takes place as $\Gd \to 0$. See section \ref{sec:ALR} for precise description of CALR with estimates. Here we highlight a few points. In dielectric case CALR occurs when $k(c)=0$ or $c=-1$ since $0$ is the accumulation point of eigenvalues. In the elasto-static case, \eqnref{kzero} is fulfilled if (and only if)
\beq
c= -\frac{\Gl+3\mu}{\Gl+\mu} \quad\text{or} \quad c= -\frac{\Gl+\mu}{\Gl+3\mu}. 
\eeq
It turns out that the cloaking region when $k(c)=k_0$ is different from that when $k(c)=-k_0$. We also mention that since $0<k_0<1/2$, \eqnref{kzero} holds only if $c<0$. The inclusion $\GO$ is assumed to be elliptic shape since eigenvalues and eigenfunctions of the NP operator can be explicitly computed. We emphasize that anomalous localized resonance does not occur on a disk since there $k_0$ and $-k_0$ are (isolated) eigenvalues of the corresponding NP operator.

\section{Spectral properties of the NP operator}\label{subsec:NP}

\subsection{Layer potentials and the NP operator}

Let $\GO$ be a bounded domain in $\Rbb^d$ ($d=2,3$) with the Lipschitz boundary.
Let $H^{1/2}(\p\GO)$ denote the usual $L^2$-Sobolev space of order $1/2$ and $H^{-1/2}(\p\GO)$ its dual space. Let $\Hcal := H^{1/2}(\p\GO)^d$ and $\Hcal^*:= H^{-1/2}(\p\GO)^d$. The duality pairing of $\Hcal^*$ and $\Hcal$ is denoted by $\la \cdot, \cdot \ra$. Let
\beq
\Psi:= \{ \Bv|_{\p\GO} \in \Hcal ~| ~\Bv=(v_1, \ldots, v_d)^T \mbox{ satisfies } \p_j v_i + \p_i v_j = 0 \mbox{ for all } 1 \leq i, j \leq d \}.
\eeq
Observe that $\Psi$ is of dimension 6 in three dimensions and 3 in two dimensions. For example, $\Psi$ in two dimensions is spanned by
\beq
\begin{bmatrix} 1 \\ 0 \end{bmatrix}, \quad
\begin{bmatrix} 0 \\ 1 \end{bmatrix}, \quad
\begin{bmatrix} y \\ -x \end{bmatrix}.
\eeq
It is worth mentioning that if $\Bv|_{\p\GO} \in \Psi$, then $\Bv$ satisfies $\Lcal_{\Gl,\Gm}\Bv = 0$ in $\GO$ and $\p_\nu \Bv = 0$ on $\p \GO$, and the converse holds.
Define
\beq
\Hcal^*_\Psi :=\{\BGvf \in \Hcal^* ~:~ \la \BGvf, \Bf \ra=0 \mbox{ for all } \Bf \in \Psi\}.
\eeq
Since $\Psi$ contains constant functions, we have in particular $\int_{\p D} \BGvf d\Gs = 0$ if $\BGvf \in \Hcal^*_\Psi$. We emphasize that if $\Lcal_{\Gl,\Gm}\Bu = 0$ in $\GO$, then $\p_\nu \Bu \in \Hcal^*_\Psi$. In fact, if $\Bf=\Bv|_{\p\GO} \in \Psi$, then
$$
\int_{\p \GO} \p_\nu \Bu \cdot \Bf d\Gs = \int_{\p \GO} \p_\nu \Bu \cdot \Bf d\Gs - \int_{\p \GO} \Bu \cdot \p_\nu \Bv d\Gs =0.
$$

The single and double layer potentials on $\p \GO$ associated with the Lam\'e parameter $(\Gl,\Gm)$ are defined by
\begin{align*}
& \BS [\BGvf] (\Bx) := \int_{\p\GO} \BGG (\Bx-\By) \BGvf(\By) d \Gs(\By), \quad \Bx \in \Rbb^d,\\
& \BD [\Bf] (\Bx) := \int_{\p\GO} \p_{\nu_\By}\BGG (\Bx-\By) \Bf(\By) d \Gs(\By), \quad \Bx \in \Rbb^d \setminus \p \GO
\end{align*}
for $\BGvf \in \Hcal^*$ and $\Bf \in \Hcal$, where $\BGG$ is the Kelvin matrix defined in \eqnref{Kelvin}. They enjoy the following jump relations \cite{DKV-Duke-88} (see also \cite{AK-book-07}):
\begin{align}
\BD [\Bf] |_{\pm} = \left(\mp \frac{1}{2}I + \BK \right) [\Bf] \quad & \mbox{a.e. on } \p \GO, \label{j-double}\\
\p_\nu \BS [\BGvf] |_{\pm} = \left(\pm \frac{1}{2} I + \BK^*\right) [\BGvf] \quad & \mbox{a.e. on } \p \GO, \label{j-single}
\end{align}
where $\BK$ is defined by \eqnref{BK}, and $\BK^*$ is the adjoint operator of $\BK$ on $L^2(\p \GO)^d$, that is,
\beq
\BK^* [\BGvf] (\Bx) := \mbox{p.v.} \int_{\p \GO} \p_{\nu_\Bx}{\bf \Gamma} (\Bx-\By) \BGvf(\By) d \Gs(\By) \quad \mbox{a.e. } \Bx \in \p \GO.
\eeq
Here and afterwards, the subscripts $+$ and $-$ indicate the limits (to $\p\GO$) from outside and inside $\GO$, respectively. The operator $\BK^*$ is also called the (elasto-static) NP operator on $\p\GO$.

The following lemma collects some facts to be used in the sequel, proofs of which can be found in \cite{CC-JMAA-07, DKV-Duke-88, Mitrea-JFAA-99}.
\begin{lemma}
\begin{itemize}
\item[\text{(i)}] $\BK$ is bounded on $\Hcal$, and $\BK^*$ is on $\Hcal^*$.
\item[\text{(ii)}] The spectrum of $\BK^*$ on $\Hcal^*$ lies in $(-1/2, 1/2]$.
\item[\text{(iii)}]  $1/2 I - \BK^*$ is invertible on $\Hcal^*_\Psi$.
\item[\text{(iv)}] $\BS$ as an operator defined on $\p\GO$ is bounded from $\Hcal^*$ into $\Hcal$.
\item[\text{(v)}] $\BS: \Hcal^* \to \Hcal$ is invertible in three dimensions.
\end{itemize}
\end{lemma}

In two dimensions $\BS$ may not be invertible. In fact, there is a bounded domain $\GO$ on which $\BS[\BGvf]=0$ on $\p\GO$ for some $\BGvf \neq 0$ (see the next subsection). It is worthwhile mentioning that there is such a domain for the Laplace operator \cite{Verch-JFA-84}.

\begin{lemma}
$\Psi$ is the eigenspace of $\BK$ on $\Hcal$ corresponding to $1/2$.
\end{lemma}

\begin{proof}
Let $\Bf \in \Psi$. Then $\Bf = \Bv|_{\p\GO}$ where $\Bv$ satisfies $\Lcal_{\Gl, \mu} \Bv =0$ in $\GO$ and $\p_\nu \Bv=0$ on $\p\GO$.
So, we have for $\Bx\in \Rbb^d \setminus \ol{\GO}$
\begin{align*}
\BD [\Bf](\Bx)
& = \int_{\p \GO} \p_{\nu_\By} \BGG(\Bx-\By) \Bf (\By) d\Gs(\By)\\
& = \int_{\p \GO}\left[\p_{\nu_\By} \BGG(\Bx-\By) \Bv(\By) - \BGG (\Bx-\By) \p_\nu\Bv(\By) \right] d\Gs(\By) = 0.
\end{align*}
So we infer from \eqnref{j-double} that
\beq\label{BK1/2}
\BK[\Bf]= \frac{1}{2} \Bf .
\eeq

Conversely, if \eqnref{BK1/2} holds, then we have from \eqnref{j-double} that $\BD [\Bf]|_- = \Bf$ and $\BD [\Bf](\Bx)=0$ for $\Bx\in \Rbb^d \setminus \ol{\GO}$. So $\p_\nu \BD [\Bf]|_- = \p_\nu \BD [\Bf]|_+ =0$. It implies that $\Bf \in \Psi$. This completes the proof.
\end{proof}

Let $N_d:=\frac{d(d+1)}{2}$, which is the dimension of $\Psi$. Let $\{ \Bf^{(j)} \}_{j=1}^{N_d}$ be a basis of $\Psi$ such that
\beq\label{basis2D}
\la \Bf^{(i)}, \Bf^{(j)} \ra =\Gd_{ij},
\eeq
where $\Gd_{ij}$ is the Kronecker's delta. Since $\p_\nu \BS[\Bf^{(j)}]|_- \in \Hcal^*_\Psi$ and $1/2 I - \BK^*$ is invertible on $\Hcal^*_\Psi$, there is a unique $\widetilde{\BGvf}^{(j)} \in  \Hcal^*_\Psi$ such that
$$
\left(\frac{1}{2} I - \BK^*\right) [\widetilde{\BGvf}^{(j)}]= \p_\nu \BS[\Bf^{(j)}]|_- = \left(- \frac{1}{2} I + \BK^*\right) [\Bf^{(j)}] .
$$
Define $\BGvf^{(j)}:= \widetilde{\BGvf}^{(j)} + \Bf^{(j)}$. Then, we have
\beq\label{Gvfeigen}
\BK^*[\BGvf^{(j)}] = \frac{1}{2} \BGvf^{(j)}.
\eeq
Moreover, we have
\beq\label{Gvfpsi}
\la \BGvf^{(j)}, \Bf^{(i)} \ra = \la \widetilde{\BGvf}^{(j)}, \Bf^{(i)} \ra + \la \Bf^{(j)}, \Bf^{(i)} \ra = \Gd_{ij},
\eeq
which, in particular, implies that $\BGvf^{(j)}$'s are linearly independent.

Let
\beq
W:= \mbox{span} \left\{ \BGvf^{(1)}, \ldots, \BGvf^{(N_d)} \right\},
\eeq
and let
\beq
\Hcal_W :=\{\Bf \in \Hcal ~:~ \la \BGvf, \Bf \ra=0 \mbox{ for all } \BGvf \in W \}.
\eeq

\begin{lemma}\label{decomplemma}
The following hold.
\begin{itemize}
\item[{\rm (i)}] Each $\BGvf \in \Hcal^*$ is uniquely decomposed as
\beq\label{decomHstar}
\BGvf= \BGvf' + \BGvf'' := \BGvf' +  \sum_{j=1}^{N_d} \la \BGvf, \Bf^{(j)} \ra \BGvf^{(j)},
\eeq
and $\BGvf' \in \Hcal^*_\Psi$.

\item[{\rm (ii)}] Each $\Bf \in \Hcal$ is uniquely decomposed as
\beq\label{decomH}
\Bf= \Bf' + \Bf'' := \Bf' + \sum_{j=1}^{N_d} \la  \BGvf^{(j)}, \Bf \ra \Bf^{(j)},
\eeq
and $\Bf' \in \Hcal_W$.

\item[{\rm (iii)}] $\BS$ maps $W$ into $\Psi$, and $\Hcal^*_\Psi$ into $\Hcal_W$.

\item[{\rm (iv)}] $W$ is the eigenspace of $\BK^*$ corresponding to the eigenvalue $1/2$.
\end{itemize}
\end{lemma}

\begin{proof}
For $\BGvf \in \Hcal^*$, and let $\BGvf''$ be as in \eqnref{decomHstar}.
Then, one can immediately see from \eqnref{Gvfpsi} that $\la \BGvf', \Bf^{(j)} \ra =0$ for all $j$, and hence $\BGvf' \in \Hcal^*_\Psi$. Uniqueness of the decomposition can be proved easily. (ii) can be proved similarly.

Thanks to \eqnref{Gvfeigen} we have $\p_\nu \BS[\BGvf^{(j)}]|_{-}=0$, and so $\BS[\BGvf^{(j)}]|_{\p\GO} \in \Psi$.
If $\BGvf \in \Hcal^*_\Psi$, then
$$
\la \BGvf^{(j)}, \BS[\BGvf] \ra = \la  \BGvf , \BS[\BGvf^{(j)}] \ra = 0
$$
for all $j$. So, $\BS$ maps $\Hcal^*_\Psi$ into $\Hcal_W$. This proves (iii).

Suppose that $\BK^*[\BGvf] = 1/2 \BGvf$ and that $\BGvf$ admits the decomposition \eqnref{decomHstar}. Then $\BK^*[\BGvf'] = 1/2 \BGvf'$. So we have from (iii) that $\BS[\BGvf'] \in \Psi$, and hence $\la \BGvf', \BS[\BGvf'] \ra=0$. Since $\int_{\p\GO} \BGvf' d\Gs=0$, we have from \eqnref{j-single}
\begin{align}
-\la \BGvf', \BS[\BGvf'] \ra & = \la \p_\nu \BS[\BGvf']|_-, \BS[\BGvf'] \ra
- \la \p_\nu \BS[\BGvf']|_+, \BS[\BGvf'] \ra \nonumber \\
&= \| \nabla \BS[\BGvf'] \|_{L^2(\GO)}^2 + \| \nabla \BS[\BGvf'] \|_{L^2(\Rbb^d \setminus \GO)}^2. \label{greenarg}
\end{align}
So $\BS[\BGvf']=\mbox{const.}$ in $\Rbb^d$. Thus we have $\BGvf' = \p_\nu \BS[\BGvf']|_+ - \p_\nu \BS[\BGvf']|_- =0$, and hence $\BGvf \in W$. Thus (iv) is proved. This completes the proof.
\end{proof}

\subsection{Symmetrization of the NP operator}

In this section we introduce a new inner product on $\Hcal^*$ (and $\Hcal$) which makes the NP operator $\BK^*$ self-adjoint.

In three dimensions, $\BS[\BGvf](\Bx)=O(|\Bx|^{-1})$ as $|\Bx| \to \infty$. Using this fact, one can show that $-\BS$ is positive-definite. In fact,  similarly to \eqnref{greenarg} we obtain
\beq
-\la \BGvf, \BS[\BGvf] \ra
= \| \nabla \BS[\BGvf] \|_{L^2(\GO)}^2 + \| \nabla \BS[\BGvf] \|_{L^2(\Rbb^d \setminus \GO)}^2 \ge 0.
\eeq
If $\la \BGvf, \BS[\BGvf] \ra=0$, then $\BS[\BGvf]$ is constant in $\Rbb^3$. Thus we have $\BGvf = \p_\nu \BS[\BGvf]|_+ - \p_\nu \BS[\BGvf]|_- =0$.
So, if we define
\beq
( \BGvf, \Bpsi )_* := -\la \BGvf, \BS[\BGvf] \ra,
\eeq
it is an inner product on $\Hcal^*$.

In two dimensions, the same argument shows that $-\BS$ is positive-definite on $\Hcal^*_\Psi$. In fact, if $\BGvf \in \Hcal^*_\Psi$, then $\BS[\BGvf](\Bx)=O(|\Bx|^{-1})$ as $|\Bx| \to \infty$, and hence we can apply the same argument as in three dimensions. However, $-\BS$ may fail to be positive on $W$: if $\GO$ is the disk of radius $r$ (centered at 0), then we have
\beq\label{appendixA}
\BS [\Bc] (\Bx) = \left[\Ga_1 r \ln r - \frac{\Ga_2 r}{2} \right] \Bc  \quad \mbox{for } \Bx \in \GO.
\eeq
for any constant vector $\Bc= (c_1 , c_2)^T$. It shows that $-\BS$ can be positive or negative depending on $r$.
To see \eqnref{appendixA}, we note that
\begin{align*}
\BS [\Bc]_i (\Bx) & = \frac{\Ga_1 c_i}{2\pi} \int_{\p \GO_r} \ln |\Bx-\By| d\Gs(\By) - \frac{\Ga_2}{2\pi} \sum_{j=1}^2 c_j \int_{\p \GO_r} \frac{(\Bx-\By)_i (\Bx-\By)_j}{|\Bx-\By|^2} d\Gs(\By)\\
& = \Ga_1 c_i \Scal[1] (\Bx) - \Ga_2 \Big(x_i \Bc \cdot \nabla \Scal[1] (\Bx) - \Bc \cdot \nabla \Scal [y_i] (\Bx)\Big),
\end{align*}
where $\Scal$ is the electro-static single layer potential, namely,
\beq\label{electrosingle}
\Scal[f] (\Bx) = \frac{1}{2\pi} \int_{\p \GO_r} \ln |\Bx-\By| f(\By) \, d\Gs(\By).
\eeq
It is known (see \cite{ACKLM-ARMA-13}) that $\Scal[1](\Bx) = r \ln r$ and $\Scal[y_i] (\Bx) = - \frac{r x_i}{2}$ for $\Bx \in \GO$. So we have \eqnref{appendixA}.

We introduce a variance of $\BS$ in two dimensions. For $\BGvf \in \Hcal^*$, define using the decomposition \eqnref{decomHstar}
\beq
\wBS[\BGvf]:= \BS[\BGvf'] + \sum_{j=1}^{3} \la \BGvf, \Bf^{(j)} \ra \Bf^{(j)}.
\eeq
We emphasize that $\wBS[\BGvf]=\BS[\BGvf]$ for all $\BGvf \in \Hcal^*_\Psi$ and $\wBS[\BGvf^{(j)}]=\Bf^{(j)}$, $j=1,2,3$.
In view of \eqnref{Gvfpsi} and Lemma \ref{decomplemma} (iii), we have
\beq
- \la \BGvf, \wBS[\BGvf] \ra = - \la \BGvf', \BS[\BGvf'] \ra + \sum_{j=1}^{3} |\la \BGvf, \Bf^{(j)} \ra|^2 .
\eeq
So, $-\wBS$ is positive-definite on $\Hcal^*$. In fact, since $- \la \BGvf', \BS[\BGvf'] \ra \ge 0$, we have $- \la \BGvf, \wBS[\BGvf] \ra \ge 0$. If $- \la \BGvf, \wBS[\BGvf] \ra = 0$, then $- \la \BGvf', \BS[\BGvf'] \ra=0$ and $\sum_{j=1}^{3} |\la \BGvf, \Bf^{(j)} \ra|^2=0$. So, $\BGvf'=0$ and $\la \BGvf, \Bf^{(j)} \ra=0$ for all $j$, and hence $\BGvf=0$.

Let us also denote $\BS$ in three dimensions by $\wBS$ for convenience. Define
\beq\label{innerstar}
( \BGvf, \Bpsi )_* := - \la \BGvf, \wBS[\Bpsi] \ra, \quad \BGvf, \Bpsi \in \Hcal^*.
\eeq

\begin{prop}
 $( \cdot, \cdot )_*$ is an inner product on $\Hcal^*$. The norm induced by $( \cdot, \cdot )_*$, denoted by $\| \cdot \|_*$, is equivalent to $\| \cdot \|_{-1/2}$.
\end{prop}

\begin{proof}
Positive-definiteness of $- \wBS$ implies that $\wBS: \Hcal^* \to \Hcal$ is bijective. So, we have
$$
\| \BGvf \|_{-1/2} \approx \| \wBS[\BGvf] \|_{1/2}.
$$
Here and throughout this paper $A \lesssim B$ means that there is a constant $C$ such that $A \le C B$, and $A \approx B$ means $A \lesssim B$ and $B \lesssim A$. It then follows from the definition \eqnref{innerstar} that
$$
| ( \BGvf, \BGvf )_*| \le \| \BGvf \|_{-1/2} \| \wBS[\BGvf] \|_{1/2} \lesssim \| \BGvf \|_{-1/2}^2.
$$

We have from the Cauchy Schwarz inequality
$$
|\la \BGvf, \wBS[\Bpsi] \ra| = |( \BGvf, \Bpsi )_*| \le \| \BGvf \|_* \| \Bpsi \|_* \lesssim \| \BGvf \|_* \| \wBS[\Bpsi] \|_{1/2}.
$$
So we have
$$
\| \BGvf \|_{-1/2} = \sup_{\Bpsi \neq 0} \frac{|\la \BGvf, \wBS[\Bpsi] \ra|}{\| \wBS[\Bpsi] \|_{1/2}} \lesssim \| \BGvf \|_*.
$$
This completes the proof.
\end{proof}

We may define a new inner product on $\Hcal$ by
\beq
( \Bf, \Bg ):= ( \wBS^{-1}[\Bf], \wBS^{-1}[\Bg] )_*= - \la \wBS^{-1}[\Bf], \Bg \ra, \quad \Bf, \Bg \in \Hcal.
\eeq

\begin{prop}\label{isometry}
$( \cdot, \cdot )$ is an inner product on $\Hcal$. The norm induced by $( \cdot, \cdot )$, denoted by $\| \cdot \|$, is equivalent to $\| \cdot \|_{1/2}$. Moreover, $\wBS$ is an isometry between $\Hcal^*$ and $\Hcal$.
\end{prop}

As shown in \cite{KPS-ARMA-07}, the NP operator $\BK^*$ can be realized as a self adjoint operator on $\Hcal^*$ using Plemelj's symmetrization principle which states for the Lam\'e system:
\beq\label{Plemelj}
\BS\BK^* = \BK \BS .
\eeq
This relation is a consequence of the Green's formula. In fact, if $\Lcal_{\Gl, \Gm} \Bu = 0$ in $\GO$, then we have for $\Bx\in \Rbb^d \setminus \ol{\GO}$
\begin{align*}
\BS \left[ \p_\nu \Bu |_{-}\right] (\Bx) - \BD [\Bu|_{-}](\Bx) = 0.
\end{align*}
Substituting $\Bu(\Bx) = \BS [\BGvf](\Bx)$ for some $\BGvf \in \Hcal^*$ into the above relation yields
$$
\BS \left(-\frac{1}{2}I + \BK^* \right) [\BGvf] (\Bx) = \BD\BS[\BGvf](\Bx),  \quad \Bx \in \Rbb^d \setminus \ol{\GO}.
$$
Letting $\Bx$ approach to $\p \GO$, we have from \eqnref{j-double}
$$
\BS \left(-\frac{1}{2}I + \BK^*\right) [\BGvf] (\Bx) = \left(-\frac{1}{2}I + \BK \right)\BS[\BGvf](\Bx), \quad \Bx \in \p \GO.
$$
So we have \eqnref{Plemelj}.

The relation \eqnref{Plemelj} holds with $\BS$ replaced by $\wBS$, namely,
\beq\label{Plemelj2}
\wBS\BK^* = \BK \wBS .
\eeq
In fact, if $\BGvf \in W$, then $\BK^*[\BGvf] = 1/2 \BGvf$ and $\wBS[\BGvf] \in \Psi$. So, we have
$$
\wBS\BK^*[\BGvf] = \BK \wBS[\BGvf].
$$
This proves \eqnref{Plemelj2}.

\begin{prop}
The NP operators $\BK^*$ and $\BK$ are self-adjoint with respect to $( \cdot, \cdot )_*$ and $( \cdot, \cdot )$, respectively.
\end{prop}
\begin{proof}
According to \eqnref{Plemelj2}, we have
\begin{align*}
( \BGvf, \BK^* [\Bpsi] )_* & = - \la \BGvf, \BS \BK^* [\Bpsi] \ra = - \la \BGvf, \BK \BS [\Bpsi] \ra \\
& = - \la \BK^* [\BGvf], \BS [\Bpsi] \ra  = ( \BK^* [\BGvf], \Bpsi)_*.
\end{align*}
So $\BK^*$ is self-adjoint. That $\BK$ is self-adjoint can be proved similarly.
\end{proof}

\section{Spectrum of NP operators on smooth planar domains}

In this section we prove \eqnref{statcompact} when $\p\GO$ is $C^{1, \Ga}$ for some $\Ga>0$. For that purpose we look into $\BK$ in more explicit form.
The definition \eqnref{kerdef} and straightforward computations show that
\beq
\p_{\nu_\By}\BGG(\Bx-\By)= -\frac{\mu}{2\mu+\Gl} \BK_1(\Bx,\By) + \BK_2(\Bx,\By),
\eeq
where
\begin{align}
\BK_1(\Bx,\By) &= \frac{\Bn_\By (\Bx-\By)^T - (\Bx-\By) \Bn_\By^T}{\Go_d |\Bx-\By|^{d}} ,  \\
\BK_2(\Bx,\By) &= \frac{\mu}{2\mu+\Gl} \frac{( \Bx-\By) \cdot \Bn_\By }{\Go_d |\Bx-\By|^d} \BI + \frac{2(\mu+ \Gl)}{2\mu+\Gl} \frac{( \Bx-\By) \cdot \Bn_\By  }{\Go_d |\Bx-\By|^{d+2}} (\Bx-\By)(\Bx-\By)^T ,
\end{align}
where $\Go_d$ is $2\pi$ if $d=2$ and $4\pi$ if $d=3$, and $\BI$ is the $d \times d$ identity matrix. Let
\beq
\BT_j [\BGvf](\Bx):= \text{p.v.} \int_{\p\GO} \BK_j(\Bx,\By) \BGvf(\By) \, d \Gs(\By), \quad \Bx \in \p\GO, \ j=1,2,
\eeq
so that
\beq\label{Kdecompo}
\BK = - \frac{\mu}{2\mu+\Gl} \BT_1 + \BT_2.
\eeq
Note that each term of $\BK_2$ has the term $( \Bx-\By) \cdot \Bn_\By$. Since $\p\GO$ is $C^{1, \Ga}$, we have
$$
|( \Bx-\By) \cdot \Bn_\By | \le C |\Bx-\By|^{1+\Ga}
$$
for some constant $C$ because of orthogonality of $\Bx-\By$ and $\Bn_\By$. So we have
$$
|\BK_2(\Bx,\By)| \le C |\Bx-\By|^{-d+1+\Ga}.
$$
So $\BT_2$ is compact on $\Hcal$ (see, for example, \cite{Folland-book}), and $\BT_1$ is responsible for non-compactness of $\BK$.

\subsection{Compactness of $\BK^2 - k_0^2 I$ and spectrum}

\begin{prop}\label{prop:compact}
Let $\GO$ be a bounded $C^{1, \Ga}$ domain in $\Rbb^2$ for some $\Ga >0$. Then $\BK^2 - k_0^2 I$ is compact on $\Hcal$.
\end{prop}

\begin{proof}
In view of \eqnref{Kdecompo}, it suffices to show that $\BT_1^2 - \frac{1}{4} I$ is compact.

In two dimensions we have
$$
\BK_1(\Bx,\By) = \frac{1}{2\pi |\Bx-\By|^2}
\begin{bmatrix}
0 & K(\Bx,\By) \\
- K(\Bx,\By) & 0
\end{bmatrix},
$$
where
$$
K(\Bx,\By):= - n_2(\By) (x_1-y_1) + n_1(\By) (x_2-y_2).
$$
Let
\beq
\Rcal[\Gvf](\Bx) = \frac{1}{2\pi} \text{p.v.} \int_{\p \GO} \frac{K(\Bx-\By)}{|\Bx-\By|^2} \Gvf(\By) \, d\Gs(\By).
\eeq
Then we have
\beq
\BT_1[\BGvf] = \begin{bmatrix}
       -\Rcal [\Gvf_2]   \\
       \Rcal[\Gvf_1]
     \end{bmatrix} .
\eeq

For $\Bx \in \p\GO$, set $\GO_\Ge := \GO \setminus B_\Ge(\Bx)$ where $B_\Ge(\Bx)$ is the disk of radius $\Ge$ centered at $\Bx$. For $\Gvf \in H^{1/2}(\p\GO)$, let $u$ be the solution to $\GD u=0$ in $\GO$ with $u=\Gvf$ on $\p\GO$. Since
$$
\text{rot} \frac{\Bx-\By}{|\Bx-\By|^2}=0, \quad \Bx \neq \By,
$$
we obtain from Stokes' formula
\begin{align*}
\Rcal[\Gvf](\Bx) & = \lim_{\Ge \to 0} \frac{1}{2\pi} \int_{\p \GO_\Ge} \frac{-(x_1-y_1)n_2(\By) + (x_2 -y_2) n_1(\By)}{|\Bx- \By|^2} \Gvf(\By) \, d\Gs(\By) \\
& = \lim_{\Ge \to 0} \frac{1}{2\pi} \int_{\GO_\Ge} \frac{-(x_1-y_1) \p_2 u(\By) + (x_2 -y_2) \p_1 u(\By)}{|\Bx- \By|^2} \, d\By.
\end{align*}
Let $v$ be a harmonic conjugate of $u$ in $\GO$ and $\psi:= v|_{\p\GO}$ so that
\beq
\psi= \Tcal[\Gvf] ,
\eeq
where $\Tcal$ is the Hilbert transformation on $\p\GO$. Then we have from divergence theorem
\begin{align*}
\Rcal[\Gvf](\Bx) &= \lim_{\Ge \to 0} \frac{1}{2\pi} \int_{\GO_\Ge} \frac{(x_1-y_1) \p_1 v(\By) + (x_2 -y_2) \p_2 v(\By)}{|\Bx- \By|^2} \, d\By \\
&= \lim_{\Ge \to 0} \frac{1}{2\pi} \int_{\p \GO_\Ge} \frac{( \Bx-\By) \cdot \Bn_\By}{|\Bx-\By|^2}  \psi(\By) \, d\Gs(\By).
\end{align*}
Observe that
$$
\frac{1}{2\pi} \int_{\p \GO_\Ge} \frac{( \Bx-\By) \cdot \Bn_\By}{|\Bx-\By|^2}  \psi(\By) \, d\Gs(\By)
$$
is the electro-static double layer potential of $\Gy$, and $\Bx \notin \GO_\Ge$. So by the jump formula of the the double layer potential (see \cite{Folland-book}), we have
$$
\lim_{\Ge \to 0} \frac{1}{2\pi} \int_{\p \GO_\Ge} \frac{( \Bx-\By) \cdot \Bn_\By}{|\Bx-\By|^2}  \psi(\By) \, d\Gs(\By) = -\frac{1}{2} \psi(\Bx) + \Kcal[\psi](\Bx),
$$
where
\beq
\Kcal[\psi](\Bx):= \frac{1}{2\pi} \int_{\p \GO} \frac{( \Bx-\By) \cdot \Bn_\By}{|\Bx-\By|^2}  \psi(\By) \, d\Gs(\By), \quad \Bx \in \p\GO.
\eeq
It is worth to mention that $\Kcal$ is the electro-static NP operator.

So far we have shown that
\beq
\Rcal[\Gvf] = -\frac{1}{2} \Tcal[\Gvf] +  \Kcal \Tcal[\Gvf] .
\eeq
Since $\Tcal$ is bounded and $\Kcal$ is compact on $H^{1/2}(\p\GO)$, we have
\beq
\Rcal = -\frac{1}{2} \Tcal  + \text{compact operator}.
\eeq
Since $\Tcal^2 = -I$, we infer that $\Rcal^2 + \frac{1}{4}I$ is compact, and so is $\BT_1^2 - \frac{1}{4}I$. This completes the proof.
\end{proof}

Since $\BK^2 - k_0^2 I$ is compact and self-adjoint, it has eigenvalues converging to $0$. The proof of Proposition \ref{prop:compact} shows that neither $\BK - k_0 I$ nor $\BK + k_0 I$ is compact, so we obtain the following theorem.

\begin{theorem}\label{thmone}
Let $\GO$ be a bounded domain in $\Rbb^2$ with $C^{1,\Ga}$ boundary for some $\Ga >0$.
\begin{itemize}
\item[(i)] The spectrum of $\BK$ on $\Hcal$ consists of eigenvalues accumulating at $k_0$ and $-k_0$, and their multiplicities are finite if they are not equal to $k_0$ or $-k_0$.
\item[(ii)] The spectrum of $\BK^*$ on $\Hcal^*$ is the same as that of $\BK$ on $\Hcal$.
\item[(iii)] The set of linearly independent eigenfunctions of $\BK$ makes a complete orthogonal system of $\Hcal$.
\item[(iv)] $\BGvf$ is an eigenfunction of $\BK^*$ on $\Hcal^*$ if and only if $\wBS[\BGvf]$ is an eigenfunction of $\BK$ on $\Hcal$.
\end{itemize}
\end{theorem}

\subsection{Spectral expansion of the fundamental solution}

Let $\{\Bpsi_j \}$ be a complete orthonormal (with respect to the inner product $( \cdot , \cdot )_*$) system of $\Hcal^*_\Psi$ consisting of eigenfunctions of $\BK^*$ on $\p\GO$ in two dimensions. Then they, together with $\BGvf^{(j)}$, $j=1,2,3$, defined in section \ref{subsec:NP}, make an orthonormal system of $\Hcal^*$. Then by Theorem \ref{thmone} (iv) $\{ \BS[\Bpsi_j] \}$ together with $\Bf^{(i)}$ is a complete orthonormal system of $\Hcal$ with respect to the inner product $( \cdot , \cdot )$.

Let $\BGG(\Bx-\By)$ be the Kelvin matrix defined in \eqnref{Kelvin}. If $\Bx \in \Rbb^2 \setminus \overline{\GO}$ and $\By \in \p\GO$, then there are (column) vector-valued functions $\Ba_j$ and $\Bb_i$ such that
\beq
\BGG(\Bx-\By) = \sum_{j=1}^\infty \Ba_j(\Bx) \BS [\Bpsi_j](\By)^T + \sum_{i=1}^3 \Bb_i(\Bx) \Bf^{(i)}(\By)^T.
\eeq
It then follows that
\begin{align*}
\int_{\p\GO} \BGG(\Bx-\By) \Bpsi_l(\By) \, d\Gs(\By) & = \sum_{j=1}^\infty \Ba_j(\Bx) \la \Bpsi_l , \BS [\Bpsi_j] \ra + \sum_{i=1}^3 \Bb_i(\Bx) \la \Bpsi_l , \Bf^{(i)} \ra \\
& = -\sum_{j=1}^\infty \Ba_j(\Bx) ( \Bpsi_j, \Bpsi_l )_* + \sum_{i=1}^3 \Bb_i(\Bx) ( \BGvf^{(i)}, \Bpsi_l )_* = -\Ba_l(\Bx).
\end{align*}
In other words, we obtain $\Ba_l(\Bx) = -\BS [\Bpsi_l](\Bx)$. Likewise one can show $\Bb_i(\Bx)= \wBS [\BGvf^{(i)}](\Bx)$.
So, we obtain
$$
\BGG(\Bx-\By) = -\sum_{j=1}^\infty \BS [\Bpsi_j](\Bx) \BS [\Bpsi_j](\By)^T + \sum_{i=1}^3 \wBS [\BGvf^{(i)}](\Bx) \Bf^{(i)}(\By)^T, \quad \Bx \in \Rbb^2 \setminus \overline{\GO}, \ \By \in \p\GO.
$$
Since both sides of above are solutions of the Lam\'e equation in $\By$ for a fixed $\Bx$, we obtain the following theorem from the uniqueness of the solution to the Dirichlet boundary value problem.

\begin{theorem}[expansion in 2D]\label{thm:add2D}
Let $\GO$ be a bounded domain in $\Rbb^2$ with $C^{1,\Ga}$ boundary for some $\Ga >0$ and let $\{\Bpsi_j \}$ be a complete orthonormal system of $\Hcal^*_\Psi$ consisting of eigenfunctions of $\BK^*$.
Let $\BGG(\Bx-\By)$ be the Kelvin matrix of the fundamental solution to the Lam\'e system. It holds that
\beq
\BGG(\Bx-\By) = - \sum_{j=1}^\infty \BS [\Bpsi_j](\Bx) \BS [\Bpsi_j](\By)^T + \sum_{i=1}^3 \wBS [\BGvf^{(i)}](\Bx) \Bf^{(i)}(\By)^T, \quad \Bx \in \Rbb^2 \setminus \overline{\GO}, \ \By \in \GO.
\eeq
\end{theorem}

In three dimensions one can prove the following theorem similarly. We emphasize that it has not been proved that the NP operator on smooth domain has a discrete spectrum.

\begin{theorem}[expansion in 3D]\label{thm:add3D}
Let $\GO$ be a bounded domain in $\Rbb^3$. Suppose that the NP operator $\BK^*$ admits eigenfunctions $\{\Bpsi_j \}$ which is a complete orthonormal system of $\Hcal^*$. It holds that
\beq
\BGG(\Bx-\By) = - \sum_{j=1}^\infty \BS [\Bpsi_j](\Bx) \BS [\Bpsi_j](\By)^T, \quad \Bx \in \Rbb^3 \setminus \overline{\GO}, \ \By \in \GO.
\eeq
\end{theorem}

The expansion formula, sometimes called an addition formula, for the fundamental solution to the Laplace operator on disks, balls, ellipses, and ellipsoids is classical and well-known. That on ellipsoids is attributed to Heine (see \cite{Dassios-book}). The formulas describe expansions of the fundamental solution to the Laplace operator in terms of spherical harmonics (balls) and ellipsoidal harmonics (ellipses). General addition formula of the fundamental solution to the Laplace operator as in Theorem \ref{thm:add2D} and Theorem \ref{thm:add3D} was found in \cite{AK14}. It shows that the addition formula is a spectral expansion by eigenfunctions of the NP operator. Above theorems extend the formula to the Kelvin matrix of the fundamental solution to the Lam\'e system.
Using explicit forms of eigenfunctions to be derived in the next subsection, one can compute the expansion formula on disks and ellipses explicitly, even though we will not write down the formulas since they are too long.

\subsection{Spectrum of the NP operator on disks and ellipses}\label{subsec:disk}

In this section we write down spectrum of the NP operator on disks and ellipses. Detailed derivation of the spectrum is presented in Appendix \ref{appendix2D}.

\medskip
Suppose that $\GO$ is a disk. The spectrum of $\BK^*$ is as follows:

\medskip
\noindent{\bf Eigenvalues}:
\beq\label{eigendisk}
\frac{1}{2},\quad - \frac{\Gl}{2 (2 \Gm + \Gl)},\quad \pm k_0.
\eeq
It is worth mentioning that the second eigenvalue above is less than $1/2$ in absolute value because of the strong convexity condition \eqnref{strongcon}.

\medskip
\noindent{\bf Eigenfunctions:}
\begin{itemize}
\item[(i)] $1/2$:
\beq\label{vectordisk1}
(1,0)^T, \quad (0,1)^T, \quad (y,-x)^T,
\eeq
\item[(ii)] $- \frac{\Gl}{2 (2 \Gm + \Gl)}$:
\beq\label{vectordisk2}
(x,y)^T,
\eeq
\item[(iii)] $k_0$:
\beq\label{vectordisk3}
\begin{bmatrix} \cos m \Gt \\ \sin m \Gt\end{bmatrix}, \quad \begin{bmatrix} - \sin m \Gt \\ \cos m \Gt\end{bmatrix}, \quad m = 2, 3, \ldots,
\eeq
\item[(iv)] $-k_0$:
\beq\label{vectordisk4}
\begin{bmatrix} \cos m \Gt \\ -\sin m \Gt\end{bmatrix}, \quad \begin{bmatrix} \sin m \Gt \\ \cos m \Gt\end{bmatrix}, \quad m=1,2, \ldots.
\eeq
\end{itemize}
We emphasize that eigenfunctions are not normalized.

\medskip

We now describe eigenvalues and eigenfunctions on ellipses. Suppose that $\GO$ is an ellipse of the form
\beq\label{ellipse}
\frac{x_1^2}{a^2} + \frac{x_2^2}{b^2} < 1, \quad a \geq b > 0.
\eeq
Put $R:=\sqrt{a^2 - b^2}$.
Then the elliptic coordinates $(\Gr,\Go)$ are defined by
\beq\label{ellcoor}
x_1 = R\cosh \Gr \cos \Go, \quad x_2 = R\sinh \Gr \sin \Go, \quad \Gr \geq 0,\ 0 \leq \Go \leq 2\pi,
\eeq
in which the ellipse $\GO$ is given by $\p\GO =\{ (\Gr, \Go)  : \Gr = \Gr_0\}$, where $\Gr_0$ is defined to be $a=R\cosh \Gr_0$ and $b = R\sinh \Gr_0$.
Define
\beq\label{hzeroGo}
h_0(\Go) := R \sqrt{\sinh^2{\rho_0} + \sin^2{\Go}}.
\eeq
To make expressions short we set
\beq
q:= (\Gl+\mu)\sinh 2\rho_0
\eeq
and
\beq\label{Genpm}
\Gg_n^\pm := \sqrt{e^{4n\rho_0}\mu^2 + (\Gl+\mu)(\Gl+3\mu) + n q (\pm2e^{2n\rho_0}\mu+nq )}.
\eeq

The spectrum of $\BK^*$ is as follows

\medskip

\noindent{\bf Eigenvalues}:
\beq
\frac{1}{2}, \quad  k_{j,n}, \ j=1, \ldots, 4,
\eeq
where
\beq\begin{split}\label{Glone}
k_{1,n}&=\frac{e^{-2n \rho_0}}{2(\Gl+2\mu)}( -q n + \Gg_n^- ), \quad n \geq 1,  \\
k_{2,n}&=\frac{e^{-2n \rho_0}}{2(\Gl+2\mu)}( q n + \Gg_n^+ ), \quad n \geq 2,  \\
k_{3,n}&=\frac{e^{-2n \rho_0}}{2(\Gl+2\mu)}( -q n -\Gg_n^- ), \quad n \geq 1,  \\
k_{4,n}&=\frac{e^{-2n \rho_0}}{2(\Gl+2\mu)}( q n - \Gg_n^+ ), \quad n \geq 1.
\end{split}
\eeq

\medskip
\noindent{\bf Eigenfunctions}:
\begin{itemize}
\item[(i)] $1/2$:
\beq
\frac{1}{h_0(\Go)}\begin{bmatrix} 1 \\ 0 \end{bmatrix}, \quad \frac{1}{h_0(\Go)}\begin{bmatrix} 0 \\ 1 \end{bmatrix}, \quad \frac{1}{h_0(\Go)}\begin{bmatrix} \left((\Gl + \Gm) e^{-2\Gr_0} - (\Gl + 3 \Gm)\right) \sin \Go \\ \left((\Gl + \Gm) e^{-2\Gr_0} + (\Gl + 3 \Gm) \right) \cos \Go \end{bmatrix},
\eeq

\item[(ii)] $k_{j,n}$, $j=1,2,3,4$:
\beq\begin{split}\label{BGvfone}
\BGvf_{1,n} & = \Bpsi_{1,n} + \frac{p_n}{k_0 + k_{1,n}}
\Bpsi_{3,n}, \quad n \ge 1, \\
\BGvf_{2,n} &= \Bpsi_{2,n}
+  \frac{p_n}{k_0 + k_{2,n}} \Bpsi_{4,n} , \quad n \ge 2, \\
\BGvf_{3,n} &=
\frac{k_0 + k_{3,n}}{p_n} \Bpsi_{1,n}
+ \Bpsi_{3,n}, \quad n \ge 1 ,\\
\BGvf_{4,n} &=
\frac{k_0 + k_{4,n}}{p_n} \Bpsi_{2,n} + \Bpsi_{4,n} , \quad n \ge 1,
\end{split}\eeq
where
\beq\label{pn}
p_n =\left(\frac{1}{2}- k_0 \right) e^{-2n\rho_0},
\eeq
and
\beq
\begin{split}\label{Bpsidef}
\Bpsi_{1,n}(\Go) &= \frac{1}{h_0(\Go)} \begin{bmatrix}
\cos{ n \Go}\\ \sin n\Go \end{bmatrix},
\quad
\Bpsi_{2,n}(\Go) = \frac{1}{h_0(\Go)} \begin{bmatrix}
-\sin{ n \Go}\\ \cos n\Go \end{bmatrix},  \\
\Bpsi_{3,n}(\Go) & = \frac{1}{h_0(\Go)} \begin{bmatrix}
\cos{ n \Go}\\ -\sin n\Go \end{bmatrix},
\quad
\Bpsi_{4,n}(\Go) = \frac{1}{h_0(\Go)} \begin{bmatrix}
\sin{ n \Go}\\ \cos n\Go \end{bmatrix}.
\end{split}
\eeq
\end{itemize}

A remark on $k_{2,1}$ in \eqnref{Glone} is in order. It is given by
$$
k_{2,1}=\frac{e^{-2 \rho_0}}{2(\Gl+2\mu)}( q + \Gg_1^+ ),
$$
where
$$
\Gg_1^+ := \sqrt{e^{4\rho_0}\mu^2 + (\Gl+\mu)(\Gl+3\mu) + q (2e^{2\rho_0}\mu+q )}.
$$
Since
\begin{equation*}
\Gm^2 e^{4\Gr_0} + (\Gl + \Gm) (\Gl + 3 \Gm) + q (2 e^{2\Gr_0} \Gm + q) = \frac{1}{4}\left[(\Gl + 3 \Gm) e^{2\Gr_0} + (\Gl + \Gm) e^{-2\Gr_0}\right]^2,
\end{equation*}
we have $\Gl_{2,1} = \frac{1}{2}$ and the corresponding eigenfunction is
\begin{align*}
\BGvf_{2,1} = h_0^{-1}(\Go) \begin{bmatrix} \left((\Gl + \Gm) e^{-2\Gr_0} - (\Gl + 3 \Gm)\right) \sin \Go \\ \left((\Gl + \Gm) e^{-2\Gr_0} + (\Gl + 3 \Gm) \right) \cos \Go \end{bmatrix}.
\end{align*}
So it is listed as an eigenfunction for $1/2$.

Let us now look into the asymptotic behavior of eigenvalues as $n \to \infty$. One can easily see from the definition \eqnref{Genpm} that
\beq
\Gg_n^\pm = \Gm e^{2 n \rho_0} \pm q n \mp \frac{(\Gl + \Gm) (\Gl + 3 \Gm) q}{2\Gm^2} n e^{- 2 n \Gr_0} + e^{- 2 n \Gr_0} O(1), \nonumber
\eeq
where $O(1)$ indicates constants bounded independently of $n$.
So one infer from \eqnref{Glone} that
\beq\begin{split}\label{asympone}
k_{1, n} &= k_0 - \frac{q}{\Gl + 2 \mu} n e^{- 2 n \rho_0} + n^2 e^{- 4 n \rho_0}O(1),  \\
k_{2, n} &= k_0 + \frac{q}{\Gl + 2 \mu} n e^{- 2 n \rho_0} + n^2 e^{- 4 n \rho_0}O(1), \\
k_{3, n} &= - k_0 - \frac{(\Gl + \Gm) (\Gl + 3 \Gm) q}{4 \Gm^2\left( \Gl + 2 \mu \right)} n e^{- 4 n \rho_0} + e^{- 4 n \rho_0} O(1), \\
k_{4, n} &= - k_0 + \frac{(\Gl + \Gm) (\Gl + 3 \Gm) q}{4 \Gm^2\left( \Gl + 2 \mu \right)} n e^{- 4 n \rho_0} + e^{- 4 n \rho_0} O(1),
\end{split}\eeq
as $n \to \infty$. In particular, we see that $k_{1, n}$ and $k_{2, n}$ converge to $k_0$ while $k_{3, n}$ and $k_{4, n}$ to $-k_0$ as $n \to \infty$. We emphasize that the convergence rates are exponential.

\section{Anomalous localized resonance and cloaking}\label{sec:ALR}

\subsection{Resonance estimates}

Let $\GO$ be a bounded domain in $\Rbb^2$ with $C^{1,\Ga}$ boundary. Let $(\Gl,\Gm)$ be the Lam\'e constants of $\Rbb^2 \setminus \GO$ satisfying the strong convexity condition \eqnref{strongcon}. Let $(\wGl, \wGm)$ be Lam\'{e} constants of $\GO$. We assume that $(\wGl, \wGm)$ is of the form
\beq
  (\wGl, \wGm) := (c + i \Gd) (\Gl, \mu),
\eeq
where $c < 0$ and $\Gd > 0$. Let $\widetilde{\Cbb}$ be the isotropic elasticity tensor corresponds to $(\wGl, \wGm)$, namely, $\widetilde{\Cbb} = ( \widetilde{C}_{ijkl} )_{i, j, k, l = 1}^2$ where
\beq
\widetilde{C}_{ijkl} := \wGl \, \Gd_{ij} \Gd_{kl} + \wGm \, ( \Gd_{ik} \Gd_{jl} + \Gd_{il} \Gd_{jk} ),
\eeq
and let $\Lcal_{\wGl, \wGm}$ and $\p_{\wGv}$ be corresponding Lam\'e operator and conormal derivative, respectively. Then we have
$\Lcal_{\wGl, \wGm} = (c + i \Gd) \Lcal_{\Gl, \Gm}$ and $\p_{\wGv}= (c + i \Gd)  \p_{\Gv}$.

Let $\Cbb_\GO$ be the elasticity tensor in presence of inclusion $\GO$ so that
$$
\Cbb_\GO = \widetilde{\Cbb} \chi_{\GO} + \Cbb \chi_{\Rbb^2 \setminus \ol{\GO}},
$$
where $\chi$ denotes the characteristic function. We consider the following transmission problem:
\beq\label{eq:negative_index}
  \begin{cases}
    \ds \nabla \cdot \Cbb_\GO \hatna \Bu = \Bf  \quad & \mbox{in } \Rbb^2, \\
    \Bu (\Bx) = O(|\Bx|^{-1})  & \mbox{as } |\Bx| \rightarrow \infty,
  \end{cases}
\eeq
where $\Bf$ is a function compactly supported in $\Rbb^2 \setminus \ol{\GO}$ and satisfies
\beq
\int_{\Rbb^2} \Bf d\Bx =0.
\eeq
This condition is necessary for a solution to \eqnref{eq:negative_index} to exist. The problem \eqref{eq:negative_index} rephrased as
\beq\label{eq:transmission}
  \begin{cases}
    \Lcal_{\Gl, \Gm} \Bu = 0 & \mbox{ in } \GO, \\
    \Lcal_{\Gl, \Gm} \Bu = \Bf & \mbox{ in } \Rbb^2 \setminus \ol{\GO}, \\
    \Bu|_- = \Bu|_+ & \text{ on } \p \GO, \\
    (c+i\Gd) \p_{\Gv} \Bu|_- = \p_\Gv \Bu|_+ & \mbox{ on } \p \GO.
  \end{cases}
\eeq

Let
\beq\label{BFBx}
  \BF(\Bx) := \int_{\Rbb^2} \BGG(\Bx - \By) \Bf(\By) d\By, \quad \Bx \in \Rbb^2.
\eeq
We seek the solution $\Bu_\Gd$ to \eqref{eq:transmission} in the form
\beq\label{eq:soltion_representation}
  \Bu_\Gd (\Bx) = \BF(\Bx) + \BS [\BGvf_\Gd] (\Bx), \quad \Bx \in \Rbb^2,
\eeq
where $\BGvf_\Gd$ is to be determined. Since $\BS [\BGvf_\Gd] (\Bx)$ is continuous across $\p\GO$, the continuity of displacement (the third condition in \eqnref{eq:transmission}) is automatically satisfies. The continuity of the traction (the fourth condition in \eqnref{eq:transmission}) leads us to
\beq
\left( c + i \delta \right) \left( \p_\Gv \BF + \p_\Gv \BS [\BGvf_\Gd]|_- \right) =  \p_\Gv \BF + \p_\Gv \BS [\BGvf_\Gd]|_+   \quad \mbox{ on } \p \GO. \nonumber
\eeq
Therefore, using the jump formula \eqref{j-single}, we have
\beq\label{pGvBF}
  \left( k_\Gd(c) I - \BK^* \right) [\BGvf_\Gd] = \p_\Gv \BF,
\eeq
where
\beq
  k_\Gd(c) := \frac{c + 1 + i \Gd}{2 \left( c - 1 + i \Gd \right)}. \label{eq:solution_density}
\eeq

Let $k_j$, $j=1,2,\ldots$, be the eigenvalues (other than $1/2$) of $\BK^*$ counting multiplicities, and let $\{\Bpsi_j \}$ be corresponding normalized eigenfunctions. Then $\{\Bpsi_j , \BGvf^{(1)}, \BGvf^{(2)}, \BGvf^{(3)} \}$ is an orthonormal system of $\Hcal^*$ and
the solution to \eqnref{pGvBF} is given by
$$
\BGvf_\Gd = \sum_{j=1}^\infty \frac{( \Bpsi_j, \p_\Gv \BF )_*}{k_\Gd(c)-k_j} \Bpsi_j + \sum_{i=1}^3 \frac{( \BGvf^{(i)}, \p_\Gv \BF )_*}{k_\Gd(c)-1/2} \BGvf^{(i)}.
$$
Since $\Bf$ is supported in $\Rbb^2 \setminus \ol{\GO}$, $\Lcal_{\Gl, \Gm} \BF=0$ in $\GO$, and hence $\p_\Gv \BF \in \Hcal^*_\Psi$. So we have
\beq\label{BGvfGd}
\BGvf_\Gd = \sum_{j=1}^\infty \frac{( \Bpsi_j, \p_\Gv \BF )_*}{k_\Gd(c)-k_j} \Bpsi_j .
\eeq

Let
\beq
E(\Bu) := \int_{\GO} \hatna \Bu : \Cbb \hatna \Bu \, d\Bx.
\eeq
Here $\BA:\BB= \sum_{i,j} a_{ij} b_{ij}$ for two matrices $\BA=(a_{ij})$ and $\BB=(b_{ij})$. We are particulary interested in estimating $\Gd E(\Bu_\Gd)$ since cloaking by anomalous localized resonance is characterized by the condition $\Gd E(\Bu_\Gd) \to \infty$ as $\Gd \to 0$. It is worth mentioning that $\Gd E(\Bu_\Gd)$ is the imaginary part of $\int_{\Rbb^2} \hatna \Bu_\Gd : \Cbb_\GO \hatna \Bu_\Gd \, d\Bx$, which represents the elastic energy of the solution.

To present results of this section let us introduce a notation. For two quantities $A_\Gd$ and $B_\Gd$ depending on $\Gd$, $A_\Gd \sim B_\Gd$ means that there are constants $C_1$ and $C_2$ independent of $\Gd \le \Gd_0$ for some $\Gd_0$ such that
$$
C_1 \le \frac{A_\Gd}{B_\Gd} \le C_2.
$$

\begin{prop}\label{prop:solest}
Let $\Bu_\Gd$ be the solution to \eqnref{eq:transmission}. It holds that
\beq\label{EBu}
E(\Bu_\Gd - \BF) \sim \sum_{j=1}^\infty \frac{|( \Bpsi_j, \p_\Gv \BF )_*|^2 }{|k_\Gd(c)-k_j|^2}  .
\eeq
\end{prop}

\begin{proof}
Using the Green's formula for Lam\'{e} operator and the jump formula \eqref{j-single}, we have
\begin{align*}
E(\Bu_\Gd - \BF) &= E(\BS [\BGvf_\Gd]) = \int_{\GO} \hatna \BS [\BGvf_\Gd] : \Cbb \hatna \BS [\BGvf_\Gd] \, d\Bx \\
&= \int_{\p \GO} \BS [\BGvf_\Gd] \cdot \p_\nu \BS [\BGvf_\Gd] \, d\Gs \\
&= - \left\la (-\frac{1}{2}I + \BK^*)[\BGvf_\Gd], \BS [\BGvf_\Gd] \right\ra = \left( (-\frac{1}{2}I + \BK^*)[\BGvf_\Gd], \BGvf_\Gd \right)_*
\end{align*}
It then follows from \eqnref{BGvfGd} that
$$
E(\Bu_\Gd - \BF) = \sum_{j=1}^\infty \frac{(1/2-k_j) |( \Bpsi_j, \p_\Gv \BF )_*|^2 }{|k_\Gd(c)-k_j|^2} .
$$
Since $k_{j,n}$ accumulates at $k_0$ or $-k_0$ and $-1/2 < k_{j,n} < 1/2$, there is a constant $C>0$ such that
\beq\label{Ck1}
C \le |\frac{1}{2}-k_j| < 1
\eeq
for all $j$. So we have \eqnref{EBu}.
\end{proof}

Note that $k_\Gd(c) \to k(c)$ as $\Gd \to 0$, where $k(c)$ be the number defined in \eqnref{kzero}.  More precisely, we have
\beq
|k_\Gd(c) - k(c)| \sim \Gd.
\eeq
So we obtain the following theorem from Proposition \ref{prop:solest}, which shows that resonance occurs at the eigenvalue as $\Gd \to 0$ at the rate of $\Gd^{-2}$.
\begin{theorem}
Let $c$ be such that $k(c)=k_j$ for some $j$ and suppose that $( \Bpsi_j, \p_\Gv \BF )_* \neq 0$. Then, we have
\beq\label{resonanceest}
E(\Bu_\Gd) \sim  \Gd^{-2}
\eeq
as $\Gd \to 0$.
\end{theorem}

\begin{proof}
We have from \eqnref{EBu}
$$
E(\Bu_\Gd - \BF) \gtrsim  \Gd^{-2} .
$$
Since $E(\BF)$ is finite, we obtain \eqnref{resonanceest}.
\end{proof}

\subsection{Anomalous localized resonance on ellipses}

Anomalous localized resonance occurs at the accumulation points of eigenvalues of the NP operator (if the accumulation points are not eigenvalues). So, we assume
\beq
k(c)=k_0 \ \text{ or } \ -k_0.
\eeq
For analysis of anomalous localized resonance we need explicit form of eigenvalues and eigenfunctions. So we assume $\GO$ is an ellipse of the form \eqnref{ellipse} so that $\GO=\{ \Gr < \Gr_0 \}$ in elliptic coordinates. We emphasize that anomalous localized resonance does not occur on disks since $\pm k_0$ are eigenvalues.

We further assume that the source function $\Bf$ is a polarizable dipole, namely,
\beq
\Bf = \BA \nabla \Gd_\Bz =
\begin{bmatrix}
  \Ba_1 \cdot \nabla \Gd_\Bz \\
  \Ba_2 \cdot \nabla \Gd_\Bz
\end{bmatrix}, \nonumber
\eeq
where $\Ba_1$ and $\Ba_2$ are constant vectors, $\BA = (\Ba_1 \ \Ba_2 )^T$, and $\Bz \in \Rbb^2 \setminus \ol{\GO}$. In this case the function $\BF$ defined by \eqnref{BFBx} is given by
\beq\label{dipolesource}
\BF(\Bx) = \BF_\Bz(\Bx) = - \left( \left( \BA \nabla_{\Bx} \right)^T \BGG(\Bx - \Bz) \right)^T.
\eeq

Let $(\Gr_\Bz, \Go_\Bz)$ be the elliptic coordinates of $\Bz$, and let $\BU(\Gt)$ be the rotation by the angle $\Gt$, namely,
\beq\label{rotation}
  \BU(\Gt) :=
  \begin{bmatrix}
    \cos \Gt & - \sin \Gt \\
    \sin \Gt & \cos \Gt
  \end{bmatrix}.
\eeq
We obtain the following theorems when $k(c)=k_0$.
\begin{theorem} \label{thm:lambda_infty_blowup}
Assume $k(c) = k_0$. Let $\Bu_\Gd$ be the solution to \eqnref{eq:transmission}.
  If $\Ba_1 \neq \BU(- \pi / 2) \Ba_2$, then we have
  \beq\label{EBuGk}
    E(\Bu_\Gd) \sim
    \begin{cases}
      \left\vert \log{\Gd} \right\vert \Gd^{- 3 + \Gr_\Bz / \Gr_0} & \text{ if } \Gr_0 < \Gr_\Bz \le 3 \Gr_0, \\
      1 & \text{ if } \Gr_\Bz > 3 \Gr_0,
    \end{cases}
  \eeq
as $\Gd \to 0$.
\end{theorem}

\begin{theorem} \label{thm:solution_bound_1}
Assume $k(c) = k_0$. Let $\Bx = (\Gr, \Go)$ in the elliptic coordinates.
Then it holds for all $\Bx$ satisfying $\Gr + \Gr_{\Bz} - 4 \Gr_0 > 0$ that
  \beq
  | \Bu_\Gd(\Bx) - \BF_{\Bz}(\Bx) | \lesssim \sum_{n = 1}^\infty \frac{e^{- n \left( \Gr + \Gr_{\Bz} - 4 \Gr_0 \right)}}{n}.
  \eeq
In particular, for any $\ol{\Gr} > 4 \Gr_0 - \Gr_{\Bz}$ there exists some $C = C_{\ol{\Gr}} > 0$ such that
  \beq\label{Gkbound}
  \sup_{\Gr \ge \ol{\Gr}} | \Bu_\Gd(\Bx) - \BF_{\Bz}(\Bx) | < C.
  \eeq
\end{theorem}

We also obtain the following theorems when $k(c)=-k_0$.

\begin{theorem} \label{thm:-lambda_infty_blowup}
  Assume that $k(c) = - k_0$.
  If $\Ba_1 \neq \BU(- \pi / 2) \Ba_2$, then we have
  \beq
    E(\Bu_\Gd) \sim
    \begin{cases}
      | \log{\Gd} |^3 \Gd^{- 5/2 + \Gr_{\Bz} / 2 \Gr_0 } & \mbox{ if } \Gr_0 < \Gr_{\Bz} \le 5 \Gr_0, \\
      1 & \mbox{ if } \Gr_{\Bz} > 5 \Gr_0,
    \end{cases}
  \eeq
  as $\Gd \to 0$.
\end{theorem}

\begin{theorem} \label{thm:solution_bound_2}
  Assume $k (c) = - k_0$.
  Let $\Bx = (\Gr, \Go) \in \Rbb^2$ in the elliptic coordinates.
  Then it holds for all $\Bx$ satisfying $\Gr + \Gr_{\Bz} - 6 \Gr_0 > 0$ that
  \beq
  | \Bu_\Gd(\Bx) - \BF_{\Bz}(\Bx) | \lesssim \sum_{n = 1}^\infty n e^{- n \left( \Gr + \Gr_{\Bz} - 6 \Gr_0 \right)}.
  \eeq
In particular, for any $\ol{\Gr} > 6 \Gr_0 - \Gr_{\Bz}$ there exists some $C = C_{\ol{\Gr}} > 0$ such that
  \beq
  \sup_{\Gr \ge \ol{\Gr}} | \Bu_\Gd(\Bx) - \BF_{\Bz}(\Bx) | < C.
  \eeq
\end{theorem}

Theorem \ref{thm:-lambda_infty_blowup} and Theorem \ref{thm:solution_bound_2} show that cloaking by anomalous localized resonance occurs when $k(c)=k_0$.  In fact, \eqnref{EBuGk} shows that $\Gd E(\Bu_\Gd) \to \infty$ if $\Gr_\Bz \le 2 \Gr_0$ and $\Gd E(\Bu_\Gd) \to 0$ if $\Gr_\Bz > 2 \Gr_0$. On the other hand, \eqnref{Gkbound} shows that $\Bu_\Gd$ bounded if $\Gr > 4 \Gr_0 - \Gr_{\Bz}$. So, if we normalize the solution by $\Bv_\Gd:= \frac{1}{\Gd E(\Bu_\Gd)} \Bu_\Gd$ so that $\Gd E(\Bv_\Gd) =1$, then $\Bv_\Gd \to 0$ in $\Gr > 4 \Gr_0 - \Gr_{\Bz}$ provided that $\Gr_\Bz \le 2 \Gr_0$. So, cloaking by anomalous localized resonance occurs and the cloaking region is $\Gr_0 < \Gr_\Bz \le 2 \Gr_0$. It is worth mentioning that this cloaking region coincides with that for the dielectric case (Laplace equation) obtained in \cite{AK14}.

Theorem \ref{thm:lambda_infty_blowup} and Theorem \ref{thm:solution_bound_1} clearly show that cloaking by anomalous localized resonance occurs when $k(c)=-k_0$, and in this case the cloaking region is $\Gr_0 < \Gr_\Bz \le 3 \Gr_0$. It is interesting to observe that the cloaking region is different from that for the case $k(c)=k_0$.

Proofs of above theorems are given in Appendix \ref{app:C}.

\section*{Acknowledgement}
We would like to thank Graeme W. Milton for pointing out to us existence of references \cite{KM-JMPS-14} and \cite{LLBW-Nature-01}.

\appendix

\section{Derivation of spectrum on disks and ellipses}\label{appendix2D}

The purpose of this section is to derive spectrum of the NP operator on disks and ellipses presented in subsections \ref{subsec:disk}.
We use complex representations of the displacement vector and traction. So, we identify $\psi = \psi_1 + i \psi_2 \in \mathbb{C}$ with the vector $\Bpsi=(\psi_1, \ \psi_2)^T$, and denote
$$
\BS[\psi] = (\BS [\Bpsi])_1 + i (\BS [\Bpsi])_2.
$$

Suppose that $\GO$ is simply connected domain in $\Rbb^2$ (bounded or unbounded), and let $\psi = \psi_1 + i \psi_2 \in H^{-1/2}(\p\GO)$. It is known (see \cite{AK-book-07, Musk-book}) that there are holomorphic functions $f$ and $g$ in $\GO$ (or in $\Cbb \setminus \ol{\GO}$) such that
\beq\label{singlecomp}
2 \Gm \BS [\psi] (z) = \Gk f(z) - z \ol{f'(z)} - \ol{g(z)}, \quad \Gk = \frac{\Gl + 3 \Gm}{\Gl + \Gm}
\eeq
in $\GO$ (or in $\Cbb \setminus \ol{\GO}$), and the conormal derivative $\p_\Gv \Bu$ is represented as
\beq\label{tractioncomplex}
\p_\Gv \BS [\psi]|_{+}  d\Gs = \left(\left(\p_\Gv \BS [\psi] \right)_1 + i \left(\p_\Gv \BS [\psi] \right)_2\right)|_{+} d\Gs = - i d \Big[f(z) + z \ol{f'(z)} + \ol{g(z)}\Big],
\eeq
where $d\Gs$ is the line element of $\p \GO$ and $d$ is the exterior derivative, namely, $d = (\p/\p z) dz + (\p/\p \bar z) d \bar z$. Here $\p \GO$ is positively oriented and $f'(z) = \p f(z) /\p z$.
Moreover, it is shown that $f$ and $g$ are obtained by
\begin{align}
f(z) & = \frac{\Gm \Ga_2}{2\pi} \int_{\p \GO} \ln (z-\Gz) \psi(\Gz) d \Gs(\Gz), \label{fdef} \\
g(z) & = - \frac{\Gm \Ga_1}{2\pi} \int_{\p \GO} \ln (z-\Gz) \ol{\psi(\Gz)} d \Gs(\Gz) - \frac{\Gm \Ga_2}{2\pi} \int_{\p \GO} \frac{\ol{\Gz} \psi(\Gz)}{z-\Gz} d\Gs(\Gz). \label{gdef}
\end{align}
It is worth mentioning that above integrals are well defined for $\psi$ satisfying $\int_{\p\GO} \psi d\Gs=0$. If $\psi$ is constant, then we take a proper branch cut of $\log(z-\Gz)$ for $z \in \Cbb \setminus \ol{\GO}$.

Let
\beq\label{Lcaldef}
\Lcal[\psi](z):= \frac{1}{2\pi} \int_{\p \GO} \ln (z-\Gz) \psi(\Gz) d \Gs(\Gz)
\eeq
so that
\beq\label{fgcomp}
f(z)= \Gm \Ga_2 \Lcal[\psi](z), \quad g(z) = -\Gm \Ga_1 \Lcal[\ol{\psi}](z) - \Gm\Ga_2 \Lcal[\ol{\Gz} \psi]'(z).
\eeq
So, \eqnref{singlecomp} can be rewritten as
\beq\label{singlecomp2}
2 \BS [\psi] (z) = \Gk \Ga_2 \Lcal[\psi](z) - \Ga_2 z \ol{\Lcal[\psi]'(z)} + \Ga_1 \ol{\Lcal[\ol{\psi}](z)}
+ \Ga_2 \ol{\Lcal[\ol{\Gz} \psi]'(z)} .
\eeq

Observe that
$$
d \Big[f(z) + z \ol{f'(z)} + \ol{g(z)}\Big] = (f'(z)+ \ol{f'(z)}) dz + (z \ol{f''(z)} + \ol{g'(z)}) d\bar{z}.
$$
So we define
\beq
\Ccal[\psi](z):= \Lcal[\psi]'(z) =\frac{1}{2\pi} \int_{\p \GO} \frac{\psi(\Gz)}{z-\Gz} d\Gs(\Gz),
\eeq
then we have
\begin{align}
\p_\Gv \BS [\psi]|_{+}  d\Gs
= - i \mu\Ga_2 \left[ \Ccal[\psi] + \ol{\Ccal[\psi]} \right] dz
 + i \left[ \mu\Ga_1 \ol{\Ccal[\ol{\psi}]} - \mu\Ga_2 \left( z \ol{\Ccal[\psi]'} - \ol{\Ccal[\ol{\Gz} \psi]'} \right) \right] d\bar{z}.
\label{tractioncomplex2}
\end{align}
We shall compute $\BS [\psi]$ and $\p_\Gv \BS [\psi]|_{+}$ for proper basis functions $\psi$.

\subsection{Disks}\label{appen:disk}

Since the spectrum of the NP operator is invariant under translation and scaling (and rotation), we may assume that $\GO$ be a unit disk. Let $\psi=(\psi_1, \psi_2)^T$ be one of the following functions
\beq
\begin{bmatrix} \cos n\Gt \\ \pm \sin n\Gt \end{bmatrix}, \quad
\begin{bmatrix} \sin n\Gt \\ \pm \cos n\Gt \end{bmatrix}, \quad n = 0,1,2, \ldots,
\eeq
or equivalently (after identifying with $\psi = \psi_1+i\psi_2$)
\beq
\Gb e^{in\Gt}, \quad n=0, \pm 1, \pm 2, \ldots,
\eeq
where $\Gb$ is either $1$ or $i$.

One can see from \eqnref{appendixA} and \eqnref{j-single} that
\beq
\BK^*[\Bc]= \frac{1}{2} \Bc
\eeq
for any constant vector $\Bc$.

If $\psi= \Gb e^{in\Gt}$ with $n \neq 0$, then one can easily see that
for $|z|>1$
\beq\label{disk1}
\Ccal[\psi](z) =
\begin{cases}
0 \quad & \mbox{for } n \geq 1,\\
\Gb z^{n-1} \quad & \mbox{for } n \leq -1.
\end{cases}
\eeq
Since $\ol{\Gz^n}=\Gz^{-n}$ for $|\Gz|=1$, we have
\beq\label{disk2}
\Ccal[\ol{\psi}](z)
= \begin{cases} \ol{\Gb} z^{-n-1} \quad &\mbox{for } n \geq 1,\\
0 & \mbox{for } n \leq -1.
\end{cases}
\eeq
One can also see that
\beq\label{disk3}
\Ccal [\ol{\Gz} \psi](z)  = \begin{cases} 0 \quad & \mbox{for } n \geq 2,\\
\Gb z^{n-2} \quad & \mbox{for } n = 1 \quad \mbox{or} \quad n \leq -1.\end{cases}
\eeq

Using \eqnref{disk1}-\eqnref{disk3} we can show that the NP eigenvalues on the disk are given by \eqnref{eigendisk} and corresponding eigenfunctions by \eqnref{vectordisk1}-\eqnref{vectordisk4}. In fact, one can see from \eqnref{tractioncomplex2} and  \eqnref{disk1}-\eqnref{disk3} that
$$
\left(\left(\p_\Gv \BS [\Gb\psi_n] \right)_1 + i \left(\p_\Gv \BS [\Gb\psi_n] \right)_2\right)
= \begin{cases}
\Gm \Ga_1 \Gb \psi_n \quad &\mbox{if } n \ge 2, \\
\Gm \Ga_1 \Gb \psi_1 - \Gm \Ga_2 \ol{\Gb} \psi_1 \quad &\mbox{if } n = 1, \\
\Gm \Ga_2 \Gb \psi_n \quad &\mbox{if } n \le -1,
\end{cases}
$$
where $\psi_n (\Gz)= \Gz^n$. Therefore, we have
\beq\label{Kstardisk}
\BK^* [\Gb\psi_n]
= \begin{cases}
(\Gm \Ga_1 -1/2) \Gb \psi_n \quad &\mbox{if } n \ge 2, \\
(\Gm \Ga_1 -1/2) \Gb \psi_1 - \Gm \Ga_2 \ol{\Gb} \psi_1 \quad &\mbox{if } n = 1, \\
(\Gm \Ga_2 -1/2) \Gb \psi_n \quad &\mbox{if } n \le -1.
\end{cases}
\eeq
Since $\Gm \Ga_1 -1/2 = k_0$ and $\Gm \Ga_2 -1/2 = -k_0$, \eqnref{Kstardisk} shows that the spectrum is as presented in subsection \ref{subsec:disk}. It is helpful to mention that $\Gb=1$ and $n=1$ in \eqnref{Kstardisk} yield the second eigenvalue in \eqnref{eigendisk} and corresponding eigenfunction.

\subsection{Ellipses}\label{appen:ell}

Suppose that $\GO =\{ (\Gr, \Go)  : \Gr < \Gr_0\}$ in elliptic coordinates as in subsection \ref{subsec:disk}. Let $(\Gr, \eta)$ be the elliptic coordinate of $z \in \Cbb \setminus \ol{\GO}$ and $(\Gr_0, \Go)$ that of $\Gz \in \p\GO$ so that
\beq
z=R\cosh (\rho + i\eta), \quad \Gz=R\cosh (\rho_0 + i\Go),
\eeq
where $R=\sqrt{a^2 - b^2}$.
It is known (see, for example, \cite{AK14}) that
\begin{align*}
\frac{1}{2\pi} \ln |z- \Gz| & = - \sum_{m=1}^{\infty} \frac{1}{m \pi} (\cosh m\rho_0 \cos m\Go e^{-m \rho} \cos m\eta + \sinh m\rho_0 \sin m\Go e^{-m\rho} \sin m\eta)\\
 & \quad + \frac{1}{2\pi} \left(\rho + \ln \left(\frac{R}{2}\right)\right), \qquad \quad \mbox{for } \rho_0 < \rho.
\end{align*}
It is convenient to write $\xi=\Gr+i\eta$ so that
\beq\label{expansion}
\frac{1}{2\pi}\ln (z-\Gz) = - \sum_{m=1}^{\infty} \frac{1}{m \pi} \cosh m(\Gr_0 + i\Go)e^{- m \xi} + \frac{1}{2\pi} \left(\xi + C \right)
\eeq
for some constant $C$ whose real part is $\ln(R/2)$.

Let $\Bpsi_{j,n}$ be functions defined in \eqnref{Bpsidef}. After complexification, they can be written as
$$
\psi = \Gb h_0(\Go)^{-1} \psi_n (\Go)
$$
where $\Gb$ is either 1 or $i$, and $\psi_n(\Go)= e^{in\Go}$. In fact, we have
\beq\label{psijm}
\begin{cases}
\psi= \Bpsi_{1, n} \mbox{ and } \ol{\psi}= \Bpsi_{3, n} \quad &\mbox{if $\Gb = 1$ and $n \ge 1$}, \\
\psi= \Bpsi_{3, -n} \mbox{ and } \ol{\psi}= \Bpsi_{1, -n} \quad &\mbox{if $\Gb = 1$ and $n \le -1$}, \\
\psi= \Bpsi_{2, n} \mbox{ and } \ol{\psi}= -\Bpsi_{4, n} \quad &\mbox{if $\Gb = i$ and $n \ge 1$}, \\
\psi= \Bpsi_{4, -n} \mbox{ and } \ol{\psi}= -\Bpsi_{2, -n} \quad &\mbox{if $\Gb = i$ and $n \le -1$}.
\end{cases}
\eeq

Let us compute $\Lcal[h_0^{-1} \psi_n]$.
Since $d\Gs(\Go)= h_0(\Go)d\Go$ on $\p\GO$, we have
\begin{align*}
\Lcal[h_0^{-1} \psi_n](z) & = \frac{1}{2\pi} \int_0^{2\pi} \ln (z-\Gz) \psi_n(\Go) d \Go \\
& = - \sum_{m=1}^{\infty} \frac{e^{- m \xi}}{m \pi}\int_0^{2\pi}  \cosh m(\Gr_0 + i\Go)e^{in\Go} d \Go + \frac{1}{2\pi} \left(\xi + C \right) \int_0^{2\pi} e^{in\Go} d \Go.
\end{align*}
Thus we have for $n \neq 0$
\beq\label{Ln}
\Lcal[h_0^{-1} \psi_n](z)  =
- \frac{e^{-|n|\xi-n\Gr_0}}{|n|} .
\eeq
We also obtain
\beq\label{Lnbar}
\Lcal[h_0^{-1} \ol{\psi_n}](z)  =
- \frac{e^{-|n|\xi+n\Gr_0}}{|n|}  ,
\eeq

Since $\p z/\p \xi= R \sinh \xi$, we have
\beq\label{Cn}
\Ccal[h_0^{-1} \psi_n](z) = \Lcal[h_0^{-1} \psi_n]'(z)  =
\frac{e^{-|n|\xi-n\Gr_0}}{R \sinh \xi} .
\eeq
and
\beq\label{Cnbar}
\Ccal[h_0^{-1} \ol{\psi_n}](z) = \Lcal[h_0^{-1} \ol{\psi_n}]'(z)  =
\frac{e^{-|n|\xi+n\Gr_0}}{R \sinh \xi} .
\eeq

Since $\Gz = R \cosh (\Gr_0 + i \Go)$ on $\p\GO$, we have
$$
\ol{\Gz} \psi_n(\Gz) = \frac{R}{2}\left[e^{\Gr_0} \psi_{n-1}(\Go) + e^{-\Gr_0} \psi_{n+1}(\Go)\right] ,
$$
and hence
$$
\Lcal[\ol{\Gz} h_0^{-1} \psi_n](z) = \frac{R}{2} e^{\Gr_0} \Lcal[h_0^{-1} \psi_{n-1}](z) + \frac{R}{2} e^{-\Gr_0} \Lcal[h_0^{-1} \psi_{n+1}](z).
$$
It then follows from \eqnref{Cn} and \eqnref{Cnbar} that
\begin{align}
\Ccal[\ol{\Gz} h_0^{-1} \psi_n](z) = \Lcal[\ol{\Gz} h_0^{-1} \psi_n]'(z) & =
\frac{e^{-|n-1|\xi-(n-2)\Gr_0} + e^{-|n+1|\xi-(n+2)\Gr_0}}{2\sinh \xi} \nonumber \\
&= \begin{cases}
\ds \frac{e^{-n(\xi+\Gr_0)} \cosh(\xi+2\Gr_0)}{\sinh \xi} \quad &\mbox{if } n \ge 1 , \\
\ds \frac{e^{n(\xi-\Gr_0)} \cosh(\xi-2\Gr_0)}{\sinh \xi} \quad &\mbox{if } n \le -1.
\end{cases}
\label{Lnbarprime}
\end{align}

Let $\psi= \Gb h_0^{-1} \psi_n$ where $n \neq 0$ and $\Gb=1, i$. According to \eqnref{singlecomp2}, we have
\begin{align*}
2 \BS [\psi] (z) = \Gb \Gk \Ga_2 \Lcal[h_0^{-1} \psi_n](z) + \ol{\Gb} \Ga_2 \left( \ol{\Lcal[\ol{\Gz} h_0^{-1} \psi_n]'(z)} - z \ol{\Lcal[h_0^{-1} \psi_n]'(z)} \right) + \Gb\Ga_1 \ol{\Lcal[h_0^{-1} \ol{\psi_n}](z)} .
\end{align*}
In view of \eqnref{Cn} and \eqnref{Lnbarprime} we have
\beq
\ol{\Lcal[\ol{\Gz} h_0^{-1} \psi_n]'(z)} - z \ol{\Lcal[h_0^{-1} \psi_n]'(z)}
= \begin{cases}
\ds \frac{e^{-n(\ol{\xi}+\Gr_0)} [\cosh(\ol{\xi}+2\Gr_0) - \cosh \xi] }{\sinh \ol{\xi}}, \quad n \ge 1, \\
\ds \frac{e^{n(\ol{\Gx} - \Gr_0)} [\cosh (\ol{\Gx} - 2 \Gr_0) - \cosh \Gx]}{\sinh \ol{\Gx}}, \quad n \le -1.
\end{cases}\nonumber
\eeq
Thus, we obtain
\beq\label{single-ellipse}
2 \BS [\psi] (z) =
\begin{cases}
\ds - \frac{\Gb (\Gk \Ga_2 e^{-n(\Gx + \Gr_0)} + \Ga_1 e^{-n(\ol{\Gx} - \Gr_0)})}{n} \\
\ds \quad\quad + \frac{\ol{\Gb} \Ga_2 e^{-n(\ol{\Gx} + \Gr_0)} [\cosh (\ol{\Gx} + 2\Gr_0) - \cosh \Gx]}{\sinh \ol{\Gx}}, \quad n \ge 1, \\
\ds \frac{\Gb (\Gk \Ga_2 e^{n(\Gx - \Gr_0)} + \Ga_1 e^{n(\ol{\Gx} + \Gr_0)})}{n} \\
\ds \quad\quad + \frac{\ol{\Gb} \Ga_2 e^{n(\ol{\Gx} - \Gr_0)} [\cosh (\ol{\Gx} - 2 \Gr_0) - \cosh \Gx]}{\sinh \ol{\Gx}}, \quad n \le -1.
\end{cases}
\eeq
Let
\beq\label{hGr}
h_\Gr(\Go)= h(\Gr, \Go) := R \sqrt{\sinh^2 \Gr + \sin^2 \Go}.
\eeq
Using the identity
\begin{equation*}
\sinh (\Gr - i \Gn) \sinh (\Gr + i \Gn) = \sinh^2 \Gr + \sin^2 \Gn = \frac{h_{\Gr}^2}{R^2},
\end{equation*}
we obtain the following lemma.

\begin{lemma}\label{ap-prop}
The single layer potentials outside the ellipse, $\Gr \ge \Gr_0$, are computed as follows: for $\Gb = 1$ or $i$ and $\Gy_n (\Go) = e^{in\Go}$, $n=1,2, \ldots$,
\begin{align*}
& \BS [\Gb h_0^{-1} \Gy_n] (z) = - \frac{\Gb}{2n}\left[\Gk \Ga_2 e^{-n(\Gr + \Gr_0)} e^{-i n \Gn} + \Ga_1 e^{-n(\Gr - \Gr_0)} e^{in\Gn}\right]\\
 & \quad + \frac{\ol{\Gb} \Ga_2 R^2}{4 h_{\Gr}^2} e^{-n(\Gr + \Gr_0)} \left[e^{in\Gn} \sinh 2 (\Gr + \Gr_0) + \frac{e^{-2\Gr_0} - e^{2\Gr}}{2}  e^{i(n+2)\Gn} + \frac{e^{-2\Gr} - e^{2\Gr_0}}{2} e^{i(n-2)\Gn}\right], \\
& \BS [\Gb h_0^{-1} \Gy_{-n}] (z) = - \frac{\Gb}{2n}\left[\Gk \Ga_2 e^{-n(\Gr - \Gr_0)} e^{-in \Gn} + \Ga_1 e^{-n(\Gr + \Gr_0)} e^{in\Gn}\right]\\
& \quad + \frac{\ol{\Gb} \Ga_2 R^2}{4 h_{\Gr}^2} e^{-n(\Gr - \Gr_0)}\left[e^{in\Gn} \sinh 2 (\Gr - \Gr_0) + \frac{e^{2\Gr_0} - e^{2\Gr}}{2} e^{i(n+2)\Gn} + \frac{e^{-2\Gr} - e^{-2\Gr_0}}{2} e^{i(n-2)\Gn}\right],
\end{align*}
where $z = R\cosh(\Gr + i \Gn)$.
\end{lemma}

As an immediate consequence of above lemma we see that there is a constant $C$ such that
\beq\label{psibound}
\left| \BS [\Bpsi_{j,n}] (z) \right| \le C \frac{e^{-n(\Gr - \Gr_0)}}{n}, \quad j=1,2,
\eeq
and
\beq\label{psibound34}
\left| \BS [\Bpsi_{j,n}] (z) \right| \le C e^{-n(\Gr - \Gr_0)}, \quad j=3,4,
\eeq
for all $n$.

If $z \in \p\GO$, namely, $\Gr=\Gr_0$, then we have
$$
\cosh \Gx = \cosh (\ol{\Gx} - 2\Gr_0).
$$
Thus we obtain for $z \in  \p \GO$
$$
2 \BS [\psi] (z) =
\begin{cases}
\ds -\Gb \frac{\Gk \Ga_2 e^{-2n\Gr_0}}{n} e^{- in\eta} - \Gb \frac{\Ga_1}{n} e^{in\eta}
+ \ol{\Gb} 2\Ga_2 \sinh 2\Gr_0 e^{-2n\Gr_0} e^{in\eta} \quad &\mbox{if } n \ge 1, \\
\ds \Gb \frac{\Gk \Ga_2 }{n} e^{in\eta} + \Gb \frac{\Ga_1 e^{2n\Gr_0}}{n} e^{-in\eta} \quad &\mbox{if } n \le -1,
\end{cases}
$$
which can be rephrased as
\beq
2 h_0^{-1} \BS [\psi] =
\begin{cases}
\ds -\frac{\Gb}{\ol{\Gb}} \frac{\Gk \Ga_2 e^{-2n\Gr_0}}{n} \ol{\psi} - \left(\frac{\Ga_1}{n}
- \frac{\ol{\Gb}}{\Gb} 2 \Ga_2 \sinh 2\Gr_0 e^{-2n\Gr_0} \right) \psi \quad &\mbox{if } n \ge 1, \\
\ds \frac{\Gk \Ga_2 }{n} \psi + \frac{\Gb}{\ol{\Gb}} \frac{\Ga_1 e^{2n\Gr_0}}{n} \ol{\psi} \quad &\mbox{if } n \le -1.
\end{cases}
\eeq
So we obtain the following lemma from \eqnref{psijm}.
\begin{lemma} \label{prop:single}
It holds that
\begin{align*}
h_0^{-1} (\Go) \BS [\Bpsi_{1,n}] (\Go) & =- \left( \frac{\Ga_1}{2 n} - \Ga_2 \sinh{2 \Gr_0} e^{- 2 n \rho_0} \right) \Bpsi_{1,n} - \frac{\Gk \Ga_2 e^{- 2 n \rho_0}}{2 n} \Bpsi_{3,n} , \\
h_0^{-1} (\Go) \BS [\Bpsi_{2,n}] (\Go) & = - \left( \frac{\Ga_1}{2 n} + \Ga_2 \sinh{2 \Gr_0} e^{- 2 n \Gr_0} \right) \Bpsi_{2,n} - \frac{\Gk \Ga_2 e^{- 2 n \Gr_0}}{2 n} \Bpsi_{4,n} ,\\
h_0^{-1} (\Go) \BS [\Bpsi_{3,n}] (\Go) & = - \frac{\Ga_1 e^{- 2 n \Gr_0}}{2 n}\Bpsi_{1,n} - \frac{\Gk \Ga_2}{2 n}\Bpsi_{3,n}, \\
h_0^{-1} (\Go) \BS [\Bpsi_{4,n}] (\Go) & = - \frac{\Ga_1 e^{- 2 n \Gr_0}}{2 n} \Bpsi_{2,n} - \frac{\Gk \Ga_2}{2 n} \Bpsi_{4,n}.
\end{align*}
\end{lemma}

Now we compute $\p_{\Gv} \BS[\Gy]|_{+}(\Go)$ using \eqnref{tractioncomplex2}. We obtain from \eqnref{Cn}
\beq\label{Cnder}
\Ccal [h_0^{-1} \Gy_n]' (z) = - \frac{e^{- |n| \Gx - n \Gr_0} (|n| \sinh \Gx + \cosh \Gx)}{R^2 \sinh^3 \Gx}.
\eeq
We also obtain from \eqnref{Lnbarprime}
\beq\label{Cnbarder}
\Ccal [\ol{\Gz} h_0^{-1} \Gy_n]'(z) =
\begin{cases}
\ds - \frac{e^{-n (\Gx + \Gr_0)} [n\cosh (\Gx + 2 \Gr_0) \sinh \Gx + \cosh 2 \Gr_0]}{R \sinh^3 \Gx}, \quad n \geq 1,\\
\ds - \frac{e^{n(\Gx - \Gr_0)} [-n \cosh (\Gx - 2 \Gr_0) \sinh \Gx + \cosh 2\Gr_0]}{R \sinh^3 \Gx}, \quad n \leq -1.
\end{cases}
\eeq

Let $\Gy = \Gb h_0^{-1} \Gy_n$. Since $d\Gs(z)= h_0(\Gn)d\Gn$, $dz = i R \sinh \Gx d \Gn$, and $d\ol{z} = - i R \sinh \ol\Gx d \Gn$ on $\p\GO$, we have from \eqnref{tractioncomplex2} that
$$
h_0(\Gn) \p_\Gv \BS [\psi]|_{+}
= R \mu\Ga_2 \sinh \Gx \left[ \Ccal[\psi] + \ol{\Ccal[\psi]} \right]
 + R \sinh \ol\Gx \left[\mu\Ga_1 \ol{\Ccal[\ol{\psi}]} - \mu\Ga_2 \left( z \ol{\Ccal[\psi]'} - \ol{\Ccal[\ol{\Gz} \psi]'} \right) \right] .
$$
We then obtain from \eqnref{Cn}, \eqnref{Cnbar}, \eqnref{Cnder}, and \eqnref{Cnbarder} (after tedious computations which we omit) that
\begin{align*}
\p_{\Gv} \BS [\Gy]|_{+} (\Gn) = \left[\Gm \Ga_1 \Gb - 2 n \Gm \Ga_2 \ol{\Gb} \sinh 2 \Gr_0 e^{-2n\Gr_0}\right] h_0^{-1} \Gy_n(\Gn) + \Gm \Ga_2 \Gb e^{-2n\Gr_0} h_0^{-1} \Gy_{-n}(\Gn),
\end{align*}
for $n \ge 1$. It is helpful to mention that the following identities are used:
\begin{align*}
\cosh \Gx \cosh \ol{\Gx} - \cosh 2 \Gr_0 = - \sinh \Gx \sinh \ol{\Gx}, \quad \cosh \Gx =  \cosh (\ol{\Gx} - 2 \Gr_0).
\end{align*}
It then follows from \eqnref{j-single} that
\begin{align}
\BK^* [\Gy](\Gn) & = \left[\left(\Gm \Ga_1 - \frac{1}{2}\right)\Gb - 2 n \Gm \Ga_2 \ol{\Gb} \sinh 2 \Gr_0 e^{-2n\Gr_0}\right] h_0^{-1} \Gy_n(\Gn) \nonumber \\
& \quad + \Gm \Ga_2 \Gb e^{-2n\Gr_0} h_0^{-1} \Gy_{-n}(\Gn). \label{BK+}
\end{align}

Similarly one can see for $n \leq - 1$ that
\begin{align*}
\p_{\Gv} \BS[\Gb h_0^{-1} \Gy_n]|_{+}(\Gn) = \Gm \Ga_1 \Gb e^{2 n \Gr_0} h_0^{-1} \Gy_{-n}(\Gn) + \Gm \Ga_2 \Gb h_0^{-1} \Gy_n(\Gn),
\end{align*}
and hence
\begin{align*}
\BK^*[\Gb h_0^{-1} \Gy_{n}](\Gn) = \Gm \Ga_1 \Gb e^{- 2 n \Gr_0} h_0^{-1} \Gy_{-n}(\Gn) +\left(\Gm \Ga_2 - \frac{1}{2}\right)\Gb h_0^{-1} \Gy_{n}(\Gn).
\end{align*}

Note that $\Gm \Ga_1 - \frac{1}{2} = k_0$. We then obtain from \eqnref{psijm} that
\begin{align}
\BK^*[\Bpsi_{1,n}] (\Go) & = \left(k_0 - 2 n \Gm \Ga_2 \sinh 2 \Gr_0 e^{-2n\Gr_0}\right) \Bpsi_{1,n} + \Gm \Ga_2 e^{-2n\Gr_0} \Bpsi_{3,n},\nonumber\\
\BK^*[\Bpsi_{3,n}] (\Go) & = \Gm \Ga_1 e^{- 2 n \Gr_0} \Bpsi_{1,n} - k_0 \Bpsi_{3,n} \label{1,3}
\end{align}
with $\Gb = 1$, and
\begin{align}\label{2,4}
\BK^* [\Bpsi_{2,n}](\Go) & = \left[k_0 + 2 n \Gm \Ga_2 \sinh 2 \Gr_0 e^{-2n\Gr_0}\right] \Bpsi_{2,n} + \Gm \Ga_2 e^{-2n\Gr_0} \Bpsi_{4,n},\nonumber\\
\BK^* [\Bpsi_{4,n}](\Go) & = \Gm \Ga_1 e^{- 2 n \Gr_0} \Bpsi_{2,n} - k_0 \Bpsi_{4,n}
\end{align}
with $\Gb = i$.

We see from \eqnref{1,3} that for each $n$ $\BK^*$ acts on the space spanned by $\Bpsi_{1,n}$ and $\Bpsi_{3,n}$ like the matrix
$$
\begin{bmatrix}
\ds k_0 - 2 n \Gm \Ga_2 \sinh 2 \Gr_0 e^{-2n\Gr_0} & \Gm \Ga_2 e^{-2n\Gr_0} \\
\Gm \Ga_1 e^{- 2 n \Gr_0} & - k_0 \end{bmatrix} .
$$
So by finding the eigenvalues and eigenvectors of this matrix, one can see that $k_{1,n}$ and $k_{3,n}$ in \eqnref{Glone} are eigenvalues and $\BGvf_{1,n}$ and $\BGvf_{3,n}$ in \eqnref{BGvfone} are corresponding eigenfunctions. One can also see from from \eqnref{2,4} that $k_{2,n}$ and $k_{4,n}$ are eigenvalues and $\BGvf_{2,n}$ and $\BGvf_{4,n}$ are corresponding eigenfunctions.

\section{Proofs of CALR} \label{app:C}

Let $k_{j,n}$ ($j=1, \ldots, 4$, $n=1, \ldots$) be eigenvalues of the NP operator given in \eqnref{Glone} and $\BGvf_{j,n}$ be corresponding eigenfunctions given in \eqnref{BGvfone}. Put $\BGf_{j, n} := \BGvf_{j, n} / \| \BGvf_{j, n} \|_*$ and
\beq
\Ga_{j, n}(\Bz) := ( \BGf_{j, n}, \p_\Gv \BF_{\Bz}  )_*. \nonumber
\eeq
Then we have from \eqnref{eq:soltion_representation} and \eqnref{BGvfGd} that
\beq
\Bu_\Gd(\Bx) - \BF_{\Bz}(\Bx) = \sum_{j = 1}^4 \sum_{n} \frac{\Ga_{j, n}(\Bz)}{k_\Gd(c) - k_{j, n}} \BS [\BGf_{j, n}] (\Bx), \label{u_delta-F_z}
\eeq
and from \eqnref{EBu} that
\beq\label{EBu10}
E(\Bu_\Gd - \BF_z) \approx \sum_{j=1}^4 \sum_{n} \frac{|\Ga_{j, n}(\Bz)|^2 }{|k_\Gd(c)-k_{j,n}|^2}  .
\eeq
We obtain from Green's formula and the jump relation \eqnref{j-single} that
$$
\Ga_{j, n}(\Bz) = - \la \BGf_{j, n}, \BS[\p_\Gv \BF_{\Bz}]  \ra = - \la \p_\Gv \BS[\BGf_{j, n}]_-, \BF_{\Bz}  \ra
= \left( - k_{j,n} + \frac{1}{2} \right) \la \BGf_{j, n}, \BF_{\Bz}  \ra
$$
So we have from \eqnref{dipolesource} that
$$
\Ga_{j, n}(\Bz) = \left( k_{j,n} - \frac{1}{2} \right) \left( \BA \nabla \right)^T \BS [\BGf_{j, n}] (\Bz),
$$
Thanks to \eqnref{Ck1} we have
\beq
|\Ga_{j, n}(\Bz)| \approx \frac{ \left| \left( \BA \nabla \right)^T \BS [\BGvf_{j, n}] (\Bz) \right|}{\| \BGvf_{j, n} \|_*}. \label{alpha_n(z)}
\eeq

We now estimate $\left( \BA \nabla \right)^T \BS [\BGvf_{j, n}] (\Bz)$. Let us compute $\left\| \BGvf_{j, n} \right\|_*$ first.
From \eqnref{Bpsidef} and Lemma \ref{prop:single}, one can easily see that
\begin{align*}
  \left\| \Bpsi_{1, n} \right\|_*^2 =  - \left\la \Bpsi_{1, n}, \BS [\Bpsi_{1, n}] \right\ra
  = \Gp \left( \frac{\Ga_1}{n} - 2 \Ga_2 \sinh{2 \Gr_0} e^{- 2 n \Gr_0} \right) .
\end{align*}
We can also see that
\begin{align*}
\left\| \Bpsi_{2, n} \right\|_*^2 &= \pi \left( \frac{\Ga_1}{n} + 2 \Ga_2 e^{- 2 n \Gr_0} \sinh{2 \rho_0} \right), \\
\left\| \Bpsi_{3, n} \right\|_*^2 &= \left\| \Bpsi_{4, n} \right\|_*^2 = \frac{\pi \Gk \Ga_2}{n}.
\end{align*}
In addition, we have
$$
\left( \Bpsi_{1, n}, \Bpsi_{3, n} \right)_* = \left( \Bpsi_{2, n}, \Bpsi_{4, n} \right)_* = \frac{\pi \Ga_1 e^{-2 n \Gr_0}}{n}.
$$
It then follows from \eqref{BGvfone} and \eqnref{pn} that
\begin{align}
  \left\| \BGvf_{1, n} \right\|_*^2 = & - \left\la \Bpsi_{1, n} + \frac{p_n}{k_0 + k_{1, n}} \Bpsi_{3, n}, \BS [\Bpsi_{1, n}] + \frac{p_n}{k_0 + k_{1, n}} \BS [\Bpsi_{3, n}] \right\ra \nonumber \\
  = & \left\| \Bpsi_{1, n} \right\|_*^2 + \left( \frac{p_n}{k_0 + k_{1, n}} \right)^2 \left\| \Bpsi_{3, n} \right\|_*^2 + \frac{2 p_n}{k_0 + k_{1, n}} \left( \Bpsi_{1, n}, \Bpsi_{3, n} \right)_* \nonumber \\
  = & \frac{\pi \Ga_1}{n} + e^{- 2 n \Gr_0} O(1). \label{norm_varphi_1}
\end{align}
In the same way, we also obtain
\begin{align}
\left\| \BGvf_{2, n} \right\|_*^2 &= \frac{\pi \Ga_1}{n} + e^{- 2 n \Gr_0} O(1), \label{norm_varphi_2} \\
\left\| \BGvf_{3, n} \right\|_*^2 &= \frac{\pi \Gk \Ga_2}{n} + n e^{- 4 n \Gr_0} O(1), \label{norm_varphi_3} \\
\left\| \BGvf_{4, n} \right\|_*^2 &= \frac{\pi \Gk \Ga_2}{n} + n e^{- 4 n \Gr_0} O(1). \label{norm_varphi_4}
\end{align}

Let us introduce two notation to make expressions short. Let $(\Gr, \Go)$ be the elliptic coordinates of $\Bz$ and let
\beq\label{bGrGo}
\Bb(\Bz) :=
  \begin{bmatrix}
    \cos \Go \sinh \Gr \\
    \sin \Go \cosh \Gr
  \end{bmatrix},
\eeq
and
\beq
 \wBU(\Bz)= (e^{2(\Gr - \Gr_0)} - e^{-2(\Gr - \Gr_0)}) \BI + (e^{2\Gr_0} - e^{2\Gr}) \BU(-2\Go) + (e^{-2\Gr} - e^{-2\Gr_0}) \BU(2\Go).
\eeq
where $\BU(\Gt)$ is the rotation by the angle $\Gt$.

\begin{lemma}\label{lem:nonsin}
  The matrix $\wBU(\Bz)$ is non-singular for any $0 \le \Go \le \Gp / 2$ and $\Gr > \Gr_0$.
\end{lemma}

\begin{proof}
Put $\Gn := \Gr - \Gr_0$. Then, we have
$$
\wBU(\Gr, \Go) = \left( e^{2 \Gn} - e^{- 2 \Gn} \right) \BI - e^{2 \rho_0} \left( e^{2 \eta} - 1 \right) \BU(- 2 \Go) - e^{- 2 \rho_0} \left(1 - e^{- 2 \eta} \right) \BU(2 \Go).
$$

Assume that $0 < \Go < \Gp / 2$. Since $\Gn>0$, we have
$$
e^{2 \Gr_0} \left( e^{2 \Gn} - 1 \right) > e^{- 2 \Gr_0} \left( 1 - e^{- 2 \Gn} \right) > 0.
$$
It means that
$$
e^{2 \Gr_0} \left( e^{2 \Gn} - 1 \right) \BU(- 2 \Go) \Ba + e^{- 2 \Gr_0} \left( 1 - e^{- 2 \Gn} \right) \BU(2 \Go) \Ba \neq c \Ba
$$
for any real number $c$ and constant vector $\Ba$, which implies that
$\wBU(\Gr, \Go) \Ba \neq 0$, and hence $\wBU(\Gr, \Go)$ is non-singular

If $\Go = 0$, one can easily show that
$$
e^{2 \Gn} - e^{- 2 \Gn} - e^{2 \Gr_0} \left( e^{2 \Gn} - 1 \right) - e^{- 2 \Gr_0} (1 - e^{- 2 \Gn}) \neq 0
$$
for any $\Gn , \Gr_0 > 0$, and hence $\wBU(\Gr, 0)$ is non-singular. Similarly one can see that $\wBU(\Gr, \pi/2)$ is non-singular.
\end{proof}

Through long but straightforward computations, which will be presented at the end of this subsection, we see that
\begin{align}
(\BA \nabla)^T \BS[\Bpsi_{1, n}] (\Bz) = & \frac{R \Ga_1 e^{- n (\Gr - \Gr_0)}}{2 h(\Gr, \Go)^2} \left( \Ba_1 + \BU(\pi / 2) \Ba_2 \right) \cdot \BU(n \Go) \Bb(\Bz) , \label{nabla,S-1} \\
(\BA \nabla)^T \BS[\Bpsi_{2, n}] (\Bz) = & \frac{R \Ga_1 e^{- n (\Gr - \Gr_0)}}{2 h(\Gr, \Go)^2} \left( \BU(- \pi / 2) \Ba_1 + \Ba_2 \right) \cdot \BU(n \Go) \Bb(\Bz), \label{nabla,S-2}
\end{align}
modulo $n e^{- n (\Gr + \Gr_0)} O(1)$, and
\begin{align}
 (\BA \nabla)^T \BS[\Bpsi_{3, n}] (\Bz) = & - \frac{n\Ga_2 R^3 e^{-n(\Gr - \Gr_0)}}{8 h(\Gr, \Go)^4} \left[ \wBU (\Gr,\Go) (\Ba_1 + \BU\left(\Gp/2\right) \Ba_2 ) \right] \cdot \BU(n\Go) \Bb(\Bz), \label{nabla,S-3} \\
  ( \BA \nabla)^T \BS [\Bpsi_{4, n}] (\Bz) = & \frac{n \Ga_2 R^3 e^{- n \left( \Gr - \Gr_0 \right)}}{8 h(\Gr, \Go)^4} \left[\wBU(\Gr, \Go) \left( \BU(-\Gp / 2) \Ba_1 + \Ba_2 \right)\right] \cdot \BU(n \Go) \Bb(\Bz), \label{nabla,S-4}
\end{align}
modulo $e^{- n \left( \Gr - \Gr_0 \right)} O(1)$.

We have from \eqref{BGvfone} that
$$
\left( \BA \nabla \right)^T \BS [\BGvf_{1, n}] (\Bz) = \left( \BA \nabla \right)^T \BS [\Bpsi_{1, n}] (\Bz) + \frac{p_n}{k_0 + k_{1,n}} \left( \BA \nabla \right)^T \BS [\Bpsi_{3, n}] (\Bz) .
$$
One can see from \eqnref{pn} and \eqnref{nabla,S-3} that
\begin{align*}
\left|\frac{p_n}{\Gk + \Gl_{1,n}} \left( \BA \nabla \right)^T \BS [\Bpsi_{3, n}] (\Bz) \right| \lesssim n e^{- n (\Gr + \Gr_0)} .
\end{align*}
So, we have from \eqnref{nabla,S-1} that
\beq\label{1234}
\left( \BA \nabla \right)^T \BS [\BGvf_{1, n}] (\Bz) = \frac{R \Ga_1 e^{- n (\Gr - \Gr_0)}}{2 h(\Gr, \Go)^2} \left( \Ba_1 + \BU(\pi / 2) \Ba_2 \right) \cdot \BU(n \Go) \Bb(\Bz)
\eeq
modulo $n e^{- n (\Gr + \Gr_0)} O(1)$. Similarly one can show using \eqnref{nabla,S-2} and \eqnref{nabla,S-4} that
\beq\label{1235}
(\BA \nabla)^T \BS[\BGvf_{2, n}] (\Bz) = \frac{R \Ga_1 e^{- n (\Gr - \Gr_0)}}{2 h(\Gr, \Go)^2} \left( \BU(- \pi / 2) \Ba_1 + \Ba_2 \right) \cdot \BU(n \Go) \Bb(\Bz),
\eeq
modulo $n e^{- n (\Gr + \Gr_0)} O(1)$.

Observe that
\begin{align*}
&\left| \left( \Ba_1 + \BU(\pi / 2) \Ba_2 \right) \cdot \BU(n \Go) \Bb(\Bz) \right|^2 + \left| \left( \BU(- \pi / 2) \Ba_1 + \Ba_2 \right) \cdot \BU(n \Go) \Bb(\Bz) \right|^2 \\
& = \left| \Ba_1 + \BU(\pi / 2) \Ba_2 \right|^2 | \Bb(\Bz) |^2 .
\end{align*}
It then follows from \eqnref{1234} and \eqnref{1235} that
$$
\left| \left( \BA \nabla \right)^T \BS [\BGvf_{1, n}] (\Bz) \right|^2 + \left|(\BA \nabla)^T \BS[\BGvf_{2, n}] (\Bz) \right|^2
 = \frac{R^2 \Ga_1^2 e^{- 2n (\Gr - \Gr_0)}}{4 h(\Gr, \Go)^4} \left| \Ba_1 + \BU(\pi / 2) \Ba_2 \right|^2 | \Bb(\Bz) |^2 ,
$$
modulo $n^2 e^{- 2n (\Gr + \Gr_0)} O(1)$. We choose constant vectors $\Ba_1$ and $\Ba_2$ so that
\beq
  \Ba_1 + \BU(\Gp / 2) \Ba_2 \neq 0.
\eeq
Then we have
$$
\left| \left( \BA \nabla \right)^T \BS [\BGvf_{1, n}] (\Bz) \right|^2 + \left|(\BA \nabla)^T \BS[\BGvf_{2, n}] (\Bz) \right|^2
\approx e^{- 2n (\Gr - \Gr_0)} ,
$$
which, together with \eqnref{norm_varphi_1} and \eqnref{norm_varphi_2}, shows
\beq
  | \Ga_{1, n}(\Bz) |^2 + | \Ga_{2, n}(\Bz) |^2 \approx n e^{- 2 n \left( \Gr - \Gr_0 \right)} + n^2 e^{-2 n \Gr} O(1), \label{eq:alpha_1+alpha_2}
\eeq
Similarly one can show that
\beq
  | \Ga_{3, n}(\Bz) |^2 + | \Ga_{4, n}(\Bz) |^2 \approx n^3 e^{- 2 n \left( \Gr - \Gr_0 \right)} + n^2 e^{-2 n \left( \Gr - \Gr_0 \right)} O(1). \label{eq:alpha_3+alpha_4}
\eeq

\medskip
\noindent{\sl Proof of Theorem \ref{thm:lambda_infty_blowup}}.
In this proof $(\Gr, \Go)$ denotes the elliptic coordinates of $\Bz$. Since
\beq\label{kGdc}
|k_\Gd(c)-k_{j,n}|^2 \approx \Gd^2 + |k(c)-k_{j,n}|^2,
\eeq
it follows from \eqnref{EBu10} that
\beq\label{EBu100}
E(\Bu_\Gd - \BF_z) \approx \sum_{j=1}^4 \sum_{n} \frac{|\Ga_{j, n}(\Bz)|^2 }{\Gd^2 + |k(c)-k_{j,n}|^2}  .
\eeq
Since $k(c)=k_0$, we see from \eqnref{asympone} that
\begin{align}
| k(c) - k_{j, n} | = c_0 n e^{- 2 n \rho_0} + n^{3/2} e^{- 3 n \rho_0} O(1), \quad j=1,2, \label{k12n}
\end{align}
where
\beq
  c_0 := \frac{q}{\Gl + 2 \mu}. \nonumber
\eeq
We also see from \eqnref{asympone} that
\beq\label{k34n}
 | k(c) - k_{j, n} | \ge C
\eeq
for some constant $C$ independent of $n$ for $j=3,4$.
It then follows from \eqref{eq:alpha_1+alpha_2} and \eqref{eq:alpha_3+alpha_4} that
\beq\label{EBu110}
E(\Bu_\Gd - \BF_z) \approx \sum_{n = 1}^\infty \frac{n e^{- 2 n \left( \Gr - \Gr_0 \right)}}{\Gd^2 + c_0^2 n^2 e^{- 4 n \Gr_0}}.
\eeq

For $0 < \Gd \ll 1$, let $N \ge 1$ be the first integer such that $\Gd > c_0 N e^{- 2 N \Gr_0}$.
Then we have
\beq
\Gd^{1 / N} \sim c_0^{1 / N} N^{1 / N} e^{- 2 \Gr_0} = e^{- 2 \Gr_0} + o(1), \nonumber
\eeq
that is,
\beq
N \sim - \frac{1}{2 \Gr_0} \log{\Gd}. \nonumber
\eeq
We then write
\beq\label{EBu111}
\sum_{n = 1}^\infty \frac{n e^{- 2 n \left( \Gr - \Gr_0 \right)}}{\Gd^2 + c_0^2 n^2 e^{- 4 n \Gr_0}} = \sum_{n \le N} + \sum_{n > N} =: I_N+ II_N.
\eeq

For the first term we have
$$
I_N \sim  \sum_{n \le N} \frac{n e^{- 2 n (\Gr - \Gr_0)}}{n^2 e^{- 4 n \Gr_0}} = \sum_{n \le N} \frac{e^{2 n (3 \Gr_0 - \Gr)}}{n} \sim \int_1^N e^{2 (3 \Gr_0 - \Gr) s} \frac{ds}s.
$$
If $\Gr_0 < \Gr < 3 \Gr_0$, then one sees using an integration by parts that
\begin{align*}
\int_1^N e^{2 (3 \Gr_0 - \Gr) s} \frac{ds}s \sim & N^{- 1} e^{2 N (3 \Gr_0 - \Gr)}
\end{align*}
as $N \to \infty$. It is easy to see that $\int_1^N e^{2 (3 \Gr_0 - \Gr) s} \frac{ds}s \sim \log N$ if $\Gr = 3 \Gr_0$, and $\int_1^N e^{2 (3 \Gr_0 - \Gr) s} \frac{ds}s \sim 1$ if
$\Gr > 3 \Gr_0$. So we obtain that
  \beq
  \sum_{n \le N} \frac{n e^{2 n (3 \Gr_0 - \Gr)}}{n^2 e^{- 4 n \Gr_0}} \sim
  \begin{cases}
    \left| \log{\Gd} \right|^{- 1} \Gd^{- (3 - \Gr / \Gr_0)}, & \text{ if } \Gr_0 < \Gr < 3 \Gr_0, \\
    \log{\left| \log{\Gd} \right|}, & \text{ if } \Gr = 3 \Gr_0, \\
    1, & \text{ if } \Gr > 3 \Gr_0.
  \end{cases} \nonumber
  \eeq
On the other hand, we have
\begin{align*}
II_N \sim \frac{1}{\Gd^2} \sum_{n > N} n e^{- 2 n (\Gr - \Gr_0)}
    \sim  \left| \log{\Gd} \right| \Gd^{\Gr / \Gr_0 - 3}
  \end{align*}
So, we have from \eqnref{EBu110} and \eqnref{EBu111}
$$
E(\Bu_\Gd - \BF_z) \sim
    \begin{cases}
      \left\vert \log{\Gd} \right\vert \Gd^{- 3 + \Gr / \Gr_0} & \text{ if } \Gr_0 < \Gr \le 3 \Gr_0, \\
      1 & \text{ if } \Gr > 3 \Gr_0
    \end{cases}
$$
Since $E(\BF_z) < \infty$, we have \eqnref{EBuGk}, and the proof is complete. \qed

\bigskip

\noindent{\sl Proof of Theorem \ref{thm:solution_bound_1}}.
Here we denote by $(\Gr, \Go)$ the elliptic coordinates of $\Bx$ and by $(\Gr_\Bz, \Go_\Bz)$ those of $\Bz$. It follows from \eqnref{u_delta-F_z} and \eqnref{kGdc} that
$$
\left| \Bu_\Gd(\Bx) - \BF_{\Bz}(\Bx) \right| \le \sum_{j = 1}^4 \sum_{n} \frac{|\Ga_{j, n}(\Bz)|}{(\Gd + |k(c) - k_{j, n}|) \| \BGvf_{j,n} \|_*} \left| \BS [\BGvf_{j, n}] (\Bx) \right|.
$$
We then have from \eqnref{k12n} and \eqnref{k34n} that
$$
\left| \Bu_\Gd(\Bx) - \BF_{\Bz}(\Bx) \right| \lesssim \sum_{j = 1}^2 \sum_{n} \frac{|\Ga_{j, n}(\Bz)|}{n e^{- 2 n \rho_0} \| \BGvf_{j,n} \|_*} \left| \BS [\BGvf_{j, n}] (\Bx) \right| + \sum_{j = 3}^4 \sum_{n} \frac{|\Ga_{j, n}(\Bz)|}{ \| \BGvf_{j,n} \|_*} \left| \BS [\BGvf_{j, n}] (\Bx) \right|.
$$
It then follows from \eqnref{norm_varphi_1}-\eqnref{norm_varphi_4}, \eqnref{psibound}, \eqnref{psibound34}, \eqnref{eq:alpha_1+alpha_2}, and \eqnref{eq:alpha_3+alpha_4}  that
\begin{align*}
\left| \Bu_\Gd(\Bx) - \BF_{\Bz}(\Bx) \right| & \lesssim \sum_{j = 1}^2 \sum_{n} \frac{n^{1/2} e^{- n \left( \Gr_\Bz - \Gr_0 \right)}}{n e^{- 2 n \rho_0} n^{-1/2}} \frac{e^{-n(\Gr - \Gr_0)}}{n} + \sum_{j = 3}^4 \sum_{n} \frac{n^{3/2} e^{- n \left( \Gr_\Bz - \Gr_0 \right)}}{n^{-1/2}} e^{-n(\Gr - \Gr_0)} \\
& \lesssim \sum_{n = 1}^\infty \frac{e^{- n \left( \Gr + \Gr_{\Bz} - 4 \Gr_0 \right)}}{n}.
\end{align*}
This completes the proof. \qed

\bigskip
\noindent{\sl Proof of Theorem \ref{thm:-lambda_infty_blowup}}.
Since $k(c) = - k_0$, we obtain from \eqnref{asympone}
\beq\label{1111}
| k(c) - k_{j, n} | \approx 1, \quad j=1,2,
\eeq
and
\beq\label{2222}
|k(c) - k_{j, n} | = \frac{q^2}{4 \Gm \left( \Gl + 2 \Gm \right)^2}  n^2 e^{- 4 n \Gr_0} + n e^{- 4 n \Gr_0} O(1), \quad j=3,4.
\eeq
The rest of the proof is similar to that of Theorem~\ref{thm:lambda_infty_blowup}. \qed

\bigskip
Theorem \ref{thm:solution_bound_2} can be proved in the same way as Theorem \ref{thm:solution_bound_1} using \eqnref{1111} and \eqnref{2222}.

\bigskip

Let us prove \eqref{nabla,S-1}-\eqref{nabla,S-4}. We prove only \eqref{nabla,S-3} which is most involved and the rest can be proved similarly.
Let
\begin{align*}
h_1(\Gr, \Go) := \frac{e^{2 ( \Gr -\Gr_0 )} - e^{- 2 ( \Gr - \Gr_0)}}{h(\Gr, \Go)^2}, \quad
h_2(\Gr, \Go) := \frac{e^{2\Gr_0} - e^{2\Gr}}{h(\Gr,\Go)^2}, \quad
h_3(\Gr, \Go) := \frac{e^{-2\Gr} - e^{-2\Gr_0}}{h(\Gr,\Go)^2},
\end{align*}
where $h(\Gr,\Go)$ is defined in \eqnref{hGr}, and let
\begin{align*}
\Bf_1(\Gr, \Go) & = \begin{bmatrix} f_{11}(\Gr, \Go) \\ f_{12}(\Gr, \Go) \end{bmatrix}  := h_1(\Gr, \Go) e^{- n (\Gr - \Gr_0 )}
 \begin{bmatrix}\cos n \Go \\ \sin n \Go \end{bmatrix},  \\
\Bf_2(\Gr, \Go) & = \begin{bmatrix} f_{21}(\Gr, \Go) \\ f_{22}(\Gr, \Go) \end{bmatrix} := h_2(\Gr, \Go) e^{- n (\Gr - \Gr_0 )}
 \begin{bmatrix} \cos (n + 2) \Go \\    \sin ( n + 2 ) \Go  \end{bmatrix},  \\
\Bf_3(\Gr, \Go) & = \begin{bmatrix} f_{31}(\Gr, \Go) \\ f_{32}(\Gr, \Go) \end{bmatrix}  := h_3(\Gr, \Go) e^{- n (\Gr - \Gr_0)}
 \begin{bmatrix} \cos (n - 2) \Go \\ \sin (n-2) \Go \end{bmatrix}.
\end{align*}
Then one can see Lemma \ref{ap-prop} that
\beq
\BS [\Bpsi_{3, n}] (\Gr, \Go) = \frac{\Ga_2 R^2}{8} \left[ \Bf_1(\Gr, \Go) + \Bf_2(\Gr, \Go) + \Bf_3(\Gr, \Go) \right] + n^{- 1} e^{- n \left( \Gr - \Gr_0 \right)} O(1). \label{S_1=F_1-F_2-F_3}
\eeq

Straightforward computations yield
\begin{align*}
\p_{\Gr} f_{11} & = \left( \p_{\Gr} h_1 - n h_1 \right) e^{-n(\Gr - \Gr_0)} \cos n \Go , \\
\p_{\Go} f_{11} & = e^{-n(\Gr - \Gr_0)} \left(\p_{\Go} h_1 \cos n \Go - n h_1  \sin n \Go \right),
\end{align*}
which can be rewritten as
\beq\label{f11}
\begin{bmatrix} \p_{\Gr} \\ \p_{\Go}\end{bmatrix} f_{11} = e^{-n(\Gr - \Gr_0)} \left( \cos n \Go \begin{bmatrix} \p_{\Gr} \\ \p_{\Go}\end{bmatrix} h_1 - nh_1  \begin{bmatrix} \cos n \Go \\ \sin n \Go \end{bmatrix} \right).
\eeq
Recall the following chain rule:
$$
\nabla = \frac{R}{h^2} \BC(\Gr,\Go) \begin{bmatrix} \p_{\Gr} \\ \p_{\Go}\end{bmatrix},
$$
where
$$
\BC(\Gr,\Go) = \begin{bmatrix} \cos \Go \sinh \Gr & - \sin \Go \cosh \Gr \\ \sin \Go \cosh \Gr & \cos \Go \sinh \Gr\end{bmatrix}.
$$
It then follows from \eqnref{f11} that
$$
\nabla f_{11}  = e^{-n(\Gr - \Gr_0)}\left[(\nabla h_1) \cos n \Go - \frac{R}{h^2} n h_1  \BC(\Gr,\Go) \begin{bmatrix} \cos n \Go \\ \sin n \Go \end{bmatrix} \right].
$$
Observe that
$$
\BC(\Gr,\Go) \begin{bmatrix} \cos n \Go \\ \sin n \Go \end{bmatrix} = \BU(n\Go) \Bb(\Gr,\Go),
$$
where $\Bb(\Gr,\Go)$ is defined in \eqnref{bGrGo}. So, we have
\beq
\nabla f_{11} = e^{-n(\Gr - \Gr_0)}\left[(\nabla h_1) \cos n \Go - \frac{R}{h^2} n h_1 \BU(n\Go) \Bb(\Gr,\Go)\right].
\eeq
Likewise, one can show that
\beq
\nabla f_{12} = e^{-n(\Gr - \Gr_0)} \left[(\nabla h_1 ) \sin n \Go + \frac{R}{h^2} n h_1 \BU(\Gp/2) \BU(n \Go) \Bb(\Gr, \Go)\right].
\eeq

It follows that
\begin{align*}
(\BA \nabla)^T \Bf_1 & = \Ba_1 \cdot \nabla (F_1)_1 + \Ba_2 \cdot \nabla (F_1)_2\nonumber\\
& = e^{-n(\Gr - \Gr_0)} \Big[ (\Ba_1 \cdot \nabla h_1) \cos n \Go - \frac{nR h_1}{h^2} \Ba_1 \cdot \BU(n \Go) \Bb(\Gr, \Go)\nonumber\\
 & \hspace{2cm} + (\Ba_2 \cdot \nabla h_1) \sin n \Go + \frac{nRh_1}{h^2} \Ba_2 \cdot \BU(\Gp/2)  \BU(n \Go) \Bb(\Gr,\Go) \Big].
\end{align*}
Note that $\p^\Ga h_j$, $\left\vert \Ga \right\vert \le 1$, $\Ga \in \Nbb^2$, $j = 1 , 2, 3$, is uniformly bounded as $n \to \infty$. Thus we have
\begin{align}\label{F-1}
(\BA \nabla)^T \Bf_1 & = \frac{nR e^{-n(\Gr - \Gr_0)} h_1}{h^2} \left[ - \Ba_1 \cdot \BU(n \Go) \Bb(\Gr, \Go) + \Ba_2 \cdot \BU(\Gp/2)  \BU(n \Go) \Bb(\Gr,\Go)\right]\nonumber\\
& = -\frac{nR e^{-n(\Gr - \Gr_0)} h_1}{h^2} \left[\Ba_1 + \BU(\Gp/2)  \Ba_2\right] \cdot \BU(n\Go) \Bb(\Gr,\Go) .
\end{align}
where the equalities hold modulo $e^{-n(\Gr - \Gr_0)} O(1)$ terms.

Similarly, one can show that
\begin{align*}
\nabla f_{21} & = e^{-n(\Gr - \Gr_0)} \Big[(\nabla h_2) \cos (n+2) \Go \\
& \quad\quad  - \frac{R h_2}{h^2} \big\{ n\BU(2\Go)\BU(n\Go) + 2 \sin (n+2) \Go \BU(\Gp/2) \big\} \Bb \Big],\\
\nabla f_{22} & = e^{-n(\Gr - \Gr_0)} \Big[ (\nabla h_2) \sin (n+2)\Go \\
& \quad\quad - \frac{Rh_2}{h^2} \BU(-\Gp/2) \big\{ n\BU(2\Go)\BU(n\Go) + 2 \cos (n+2) \Go \BI \big\} \Bb \Big],
\end{align*}
and
\begin{align*}
\nabla f_{31} & = e^{-n(\Gr - \Gr_0)} \Big[ (\nabla h_3) \cos (n-2) \Go \\
& \quad\quad - \frac{R h_3}{h^2} \big\{n \BU(-2\Go)\BU(n\Go) - 2 \sin (n-2) \Go \BU(\Gp/2) \big\} \Bb \Big],\\
\nabla f_{32} & = e^{-n(\Gr - \Gr_0)} \Big[ (\nabla h_3) \sin (n-2) \Go \\
& \quad\quad - \frac{R h_3}{h^2} \BU (-\Gp/2) \big\{ n\BU(-2\Go)\BU(n\Go) + 2 \cos (n-2)\Go \BI \big\} \Bb \Big].
\end{align*}
So, we have
\begin{align}
(\BA \nabla)^T \Bf_2 & = - \frac{nR e^{-n(\Gr - \Gr_0)} h_2 }{h^2} \BU(-2\Go) \left[ \Ba_1 + \BU(\Gp/2) \Ba_2\right] \cdot \BU(n\Go) \Bb(\Gr,\Go), \label{F-2}\\
(\BA \nabla)^T \Bf_3 & = - \frac{nR e^{-n(\Gr - \Gr_0)} h_3 }{h^2} \BU(2\Go) \left[ \Ba_1 + \BU(\Gp/2) \Ba_2\right] \cdot \BU(n\Go) \Bb(\Gr,\Go) . \label{F-3}
\end{align}

Note that
$$
h_1 \BI + h_2 \BU(-2\Go) + h_3 \BU(2\Go) = h^2 \wBU.
$$
Since
\begin{align*}
(\BA \nabla)^T \BS[\Bpsi_{3,n}] = \frac{\Ga_2 R^2}{8} \left( (\BA \nabla)^T\Bf_1 + (\BA \nabla)^T\Bf_2 + (\BA \nabla)^T\Bf_3 \right) + n^{- 1} e^{- n \left( \Gr - \Gr_0 \right)} O(1),
\end{align*}
we obtain \eqref{nabla,S-3} from \eqref{F-1}, \eqref{F-2}, and \eqref{F-3}.


\end{document}